\documentclass[18pt]{siamart250211}

\usepackage[english]{babel}
\usepackage{color}
\usepackage{amsmath}
\usepackage{amsfonts}
\usepackage{graphicx}
\usepackage{algorithmic}
\usepackage{comment}
\usepackage[normalem]{ulem}
\usepackage{multirow}
\usepackage{booktabs}
\usepackage{indentfirst}
\usepackage{enumitem}

\DeclareMathOperator{\spann}{span}

\title{Mixed Precision Orthogonalization-Free Projection Methods for Eigenvalue and Singular Value Problems}
\author{
Tianshi Xu\thanks{Department of Mathematics, Emory University, Atlanta, GA, 
  (\email{tianshi.xu@emory.edu, yxi26@emory.edu}). Research of Y. Xi is supported by NSF DMS 2208412.}
   \and Zechen Zhang\thanks{Department of Computer Science and Engineering, University of Minnesota, Minneapolis, MN, 
  (\email{zhan5260@umn.edu,saad@umn.edu}). Research supported by the NSF award DMS 2208456.}
  \and Jie Chen\thanks{MIT-IBM Watson AI Lab, IBM Research.
  (\email{chenjie@us.ibm.com})}
\and Yousef Saad\footnotemark[2]
\and Yuanzhe Xi\footnotemark[1]
}
\usepackage{amsopn}

\headers{Mixed Precision Orthogonalization-Free RR}{T. Xu, Z. Zhang, J. Chen, Y. Saad, Y. Xi} 

\begin{document}
\maketitle
\begin{abstract}
Mixed-precision arithmetic offers significant computational advantages for large-scale matrix computation tasks, yet preserving accuracy and stability in eigenvalue problems and the singular value decomposition (SVD) remains challenging. This paper introduces an approach that eliminates orthogonalization requirements in traditional Rayleigh-Ritz projection methods. 
The proposed method employs non-orthogonal bases computed at reduced precision, resulting in bases computed without inner-products. A primary focus is on maintaining the linear independence of the basis vectors.
Through extensive evaluation with both synthetic test cases and real-world applications, we demonstrate that the proposed approach achieves the desired accuracy while fully taking advantage of mixed-precision arithmetic.
\end{abstract}



%
%

\section{Introduction}\label{sec1}

Recent research on themes related to high-performance computing indicates a strong surge of interest in low-precision and mixed-precision arithmetic.  There are compelling performance incentives to work with lower precision. 
Two related benefits of working in reduced precision are the \emph{lower energy consumption} and the \emph{lighter storage requirement}, which also leads to less communication \cite{anzt2019adaptive}.  
Using mixed-precision arithmetic has long been an effective approach in many areas. 
Scientists and engineers in scientific computing have traditionally leaned toward double-precision arithmetic by default, but this is now being questioned as more studies are being undertaken, and a better understanding is emerging on the impact of low precision on common computations, see, e.g., \cite{abdelfattah2019progressive,anzt2019adaptive,haidar2019investigating}.

Mixed-precision arithmetic provides notable computational benefits, yet maintaining accuracy when using reduced precision remains challenging. In this paper, we analyze the impact of mixed-precision arithmetic on eigenvalue and singular value computations. Eigenvalue decomposition and Singular Value Decompositions (SVD) constitute fundamental numerical linear algebra kernels, which are widely used in diverse scientific and data science applications~\cite{chang2005svd,guo2015efficient,halko2011finding,evsl,xu2021pargemslr,shi2018computing,shi2022non,xi2018fast}. For large-scale problems, these decompositions are typically replaced by partial eigenvalue or singular value problems where a few eigenvalues or singular values need to be computed along with their associated eigenvectors or singular vectors. This is often achieved via 
projection-type methods
\cite{parlett1973geometric,lanczos1950iteration,arnoldi1951principle}. These methods rely on matrix-vector multiplications (MatVecs) and other simple linear algebra operations to build a suitable subspace and then extract eigenvalue and singular value approximations from this subspace \cite{saad2011numerical}. In this paper, we first identify and discuss the specific challenges associated with  the use of  mixed-precision arithmetic and then propose effective strategies to address them.

In the past few years, researchers have already explored a few methods to mitigate accuracy loss in mixed-precision computations of eigenvalues and singular values. One popular approach relies on  iterative refinement originating from Newton's method for mixed-precision refinement in standard eigenvalue computations~\cite{dongarra1980improving,dongarra1983improving}. This strategy has been effectively extended to symmetric eigenvalue problems~\cite{ogita2018iterative,petschow2014improved,tsai2022mixed} and the SVD~\cite{ogita2020iterative}. Another direction of research involves identifying optimal precision levels for storing matrices without compromising performance~\cite{ooi2020effect}. Recent studies also recommend using probabilistic error analysis to determine precision requirements throughout computational phases, emphasizing high-precision reorthonormalization for maintaining accuracy~\cite{higham2022mixed}. One such example is the mixed-precision, single-pass Nystr\"{o}m method proposed in \cite{carson2024single}, which executes computationally intensive matrix multiplication operations at lower precision. 

In this paper, we focus on the Rayleigh-Ritz (RR) projection framework~\cite{saad2011numerical} for efficiently computing subsets of eigenvalues and singular values. The RR projection method involves two main stages: the first stage constructs a subspace basis encapsulating essential matrix information, typically an orthogonal basis derived via QR factorization; the second stage projects the original matrix onto this subspace, extracting the targeted eigenvalue or singular value approximations. Since classical QR-based orthogonalization methods can suffer from substantial orthogonality loss when performed in low precision, the effectiveness of the RR projection method will also be undermined in this case. To address this issue, we introduce a refined RR projection approach to improve both accuracy and computational performance with mixed precision arithmetic.

Our main contributions are the following:

\vspace{0.5\baselineskip}

\begin{enumerate}[label=\textbf{\arabic*}.]
\setlength{\itemsep}{0.5\baselineskip}

\item We introduce the Orthogonalization-Free Rayleigh-Ritz (OFRR) procedure, specifically designed to address the challenges of maintaining numerical accuracy in eigenvalue and singular value computations performed with mixed-precision arithmetic. 
Traditional approaches, reliant on QR-based orthogonalization, often suffer significant accuracy degradation in low-precision environments. 
To overcome this limitation, OFRR eliminates the explicit orthogonalization step, enabling the extraction of accurate spectral information from non-orthogonal basis vectors.

\item {We investigate the use of various approaches for generating non-orthogonal bases for the proposed OFRR procedure. Furthermore, we show that the Hessenberg process—a variant of LU factorization—outperforms Gram–Schmidt orthogonalization due to its inner-product-free feature.}

\item To evaluate the performance and accuracy of the OFRR algorithm, we conduct extensive numerical experiments using a diverse set of matrices. 
These includ challenging real-world problems from the SuiteSparse Matrix Collection \cite{davis2011university}, as well as kernel matrices arising in Gaussian processes \cite{posterior,higp}. 
The results indicate that OFRR significantly enhances approximation accuracy compared to traditional approaches that rely on orthogonalization steps.  In addition, we implement OFRR on GPU architectures to assess its practical performance. The GPU-accelerated OFRR implementation highlights the algorithm's scalability and applicability in large-scale matrix computations.
\end{enumerate}

The remaining sections are organized as follows. In Section~\ref{sec2}, we use subspace iteration as an illustrative example to demonstrate the challenges posed by low-precision arithmetic in eigenvalue computations. We then introduce the Orthogonalization-Free Rayleigh-Ritz (OFRR) procedure in Section~\ref{sec3} and examine several strategies for generating non-orthogonal bases in Section \ref{sec4}. The effectiveness and accuracy of the proposed OFRR algorithm are verified through extensive numerical experiments in Section \ref{sec5}, and concluding remarks are drawn in Section \ref{sec6}.

Following \texttt{MATLAB} syntax,
we use subscripts to access elements and submatrices of matrices and vectors. For a matrix $\mathbf{A}$, the notation $\mathbf{A}_{i,:}$ represents the entire $i$-th row, while $\mathbf{A}_{:,j}$ represents the entire $j$-th column and $\mathbf{A}_{i,j}$ is the 
entry at the $i$-th row and $j$-th column.  Similarly, for a vector $\mathbf{v}$, $\mathbf{v}_i$ indicates the $i$-th entry.  More generally, for integers $p \le q$ and $r \le s$,  $\mathbf{A}_{p:q,r:s}$ denotes the submatrix of $\mathbf{A}$ consisting of rows $p$ through $q$ and columns $r$ through $s$. The colon `$:$' in a subscript indicates selecting all indices along that dimension. If $\mathbf{p}$ is a permutation vector, then, $\mathbf{A}_{\mathbf{p}, j}$ denotes the $j$-th column of $\mathbf{A}$ with its entries permuted according to $\mathbf{p}$.
Furthermore, we let $\mathbf{e}_i$ denote the $i$-th column of an identity matrix. Finally, we represent a general subspace by $\mathcal{K}$ and use $\mathcal{K}_m(\mathbf{v},\mathbf{A})$ to denote the $m$-th Krylov subspace: 
\begin{equation*}
    \mathcal{K}_m(\mathbf{v},\mathbf{A}) := \spann\{\mathbf{v},\mathbf{A}\mathbf{v},\cdots,\mathbf{A}^{m-1}\mathbf{v}\}.
\end{equation*}

%
%

\section{Challenges in Low Precision Eigenvalue Computations}\label{sec2}
In this section, we use the subspace iteration with Rayleigh–Ritz (RR) projection as an example to identify the difficulties that contribute to the accuracy loss in low-precision eigenvalue computations. 

Subspace iteration is widely used to approximate the dominant eigenpairs of a matrix $\mathbf{A} \in \mathbb{R}^{n \times n}$. This algorithm, akin to a `block' version of the power method, begins with a randomly chosen initial set of vectors $\mathbf{X}_0 \in \mathbb{R}^{n \times k}$. Each iteration applies a power of $\mathbf{A}$ to $\mathbf{X_0}$ 
\begin{equation} \mathbf{X}_{iter} = \mathbf{A}^{iter}\mathbf{X}_0,\end{equation}
where the power $iter$ is typically fixed a priori by the user or determined 
dynamically.  
To ensure numerical stability, column scaling should follow each matrix-vector multiplication to prevent overflow or underflow. In addition, the QR algorithm is typically applied to $X_{iter}$
to preserve linear independence among the vectors. See Algorithm~\ref{alg:multistep subspace iteration} for a summary of this procedure.

\begin{algorithm}[htbp]
\caption{\textit{Multiple Step Subspace Iteration}}\label{alg:multistep subspace iteration}
\begin{algorithmic}[1]
\STATE$\triangleright$ \textbf{input:} $\mathbf{A}\in \mathbb{R}^{n\times n}$, $k$, $m$, and $iter$
\STATE$\triangleright$ \textbf{output:} $\mathbf{X_{0}}$ \hfill
\STATE$\triangleright$ Generate a set of random vectors $\mathbf{X}_0 \in \mathbb{R}^{n\times k}$
\FOR{$i=1:m$}
    \STATE$\triangleright$ Compute $\mathbf{X}_{iter} = \mathbf{A}^{iter} \mathbf{X}_0$
    \STATE$\triangleright$ Perform QR factorization $\mathbf{X}_{iter} = \mathbf{Q}\mathbf{R}$  \label{alg:multistep subspace iteration line qr}
    \STATE$\triangleright$ Set $\mathbf{X}_0=\mathbf{Q}$
    \STATE$\triangleright$ Update $iter$
\ENDFOR
\end{algorithmic}
\end{algorithm}

{
Algorithm~\ref{alg:multistep subspace iteration} generates an orthonormal basis $\mathbf{Q}$ intended to approximate the dominant invariant subspace of $\mathbf{A}$. It is worth noting that alternative approaches exist for approximating a few
eigenvalues and vectors of a
matrix. Most prominent among these is  the family of Krylov subspace methods (e.g., the Lanczos algorithm for symmetric $\mathbf{A}$ or the Arnoldi process for non-symmetric $\mathbf{A}$). Krylov methods are usually faster for such tasks,  see Section~\ref{sec:krylov}, 
but subspace iteration has a number of other advantages when the goal is to compute an invariant subspace, e.g.,
in electronic structure calculations \cite{chebfsi}.}


Algorithm \ref{alg:multistep subspace iteration} is always used in conjunction with  RR projection step (Algorithm \ref{alg:rayleigh-ritz eigenvalue}).  This modification involves updating $\mathbf{X}_0$ in Line~7 of Algorithm \ref{alg:multistep subspace iteration} using the output $\tilde{\mathbf{U}}$ from Algorithm \ref{alg:rayleigh-ritz eigenvalue}. The inputs for Algorithm \ref{alg:rayleigh-ritz eigenvalue} are the matrix $\mathbf{A}$ and the matrix $\mathbf{Q}$, which is generated in Line~6 of Algorithm \ref{alg:multistep subspace iteration}.

\begin{algorithm}[htbp]
\caption{\textit{Rayleigh-Ritz Projection with Orthogonal Bases}}\label{alg:rayleigh-ritz eigenvalue}
\begin{algorithmic}[1]
\STATE$\triangleright$ \textbf{input:} $\mathbf{A}\in \mathbb{R}^{n\times n}$, $\mathbf{Q}\in \mathbb{R}^{n\times k}$ with orthonormal columns
\STATE$\triangleright$  \textbf{output:} $\tilde{\mathbf{\Lambda}}$ and $\tilde{\mathbf{U}}$ \hfill \COMMENT{Approximate eigenvalues and Schur vectors}
\STATE$\triangleright$ Compute $\mathbf{B} = \mathbf{Q}^{\top}\mathbf{A}\mathbf{Q}$\label{alg:rayleigh-ritz eigenvalue eigenvalue}
\STATE$\triangleright$ Compute Schur decomposition $\mathbf{B}\mathbf{Y} = \mathbf{Y}\tilde{\mathbf{\Lambda}}$
\STATE$\triangleright$ Compute $\tilde{\mathbf{U}} = \mathbf{Q}\mathbf{Y}$.
\end{algorithmic}
\end{algorithm}


As illustrated in Algorithms \ref{alg:multistep subspace iteration} and \ref{alg:rayleigh-ritz eigenvalue}, subspace iteration with RR projection primarily relies on two fundamental linear algebra operations: MatVecs and vector orthogonalization. In reduced-precision environments, maintaining the orthonormality of $\mathbf{Q}$ as required by classical Rayleigh-Ritz projection presents significant challenges. To investigate the impact of orthogonality loss on eigenvalue approximations, we conducted a series of experiments using Gaussian kernel matrices defined by 
\begin{equation}
    \mathbf{A}_{ij}= f(\exp(-\Vert\mathbf{x}_i-\mathbf{x}_j\Vert_2^2 / (2l^2))+s\delta_{ij}),
    \label{eq:guassiankernel}
\end{equation}
where $\mathbf{x}_i$ and $\mathbf{x}_j$ in $\mathbb{R}^d$ are the $i$-th and $j$-th data points, respectively, from a dataset $\mathbf{D} \in \mathbb{R}^{n \times d}$. Here, $f$ represents the scale parameter, $l$ represents the length scale parameter, $s$ is the variance parameter, and $\delta_{ij}$ is a Kronecker delta function that is $1$ when $i = j$ and $0$ otherwise. We uniformly sampled $1000$ data points from a square area with side length $\sqrt{1000}$, setting $f=1$, $l=10$, and $s=0.01$ to generate a test matrix $\mathbf{A}$. 

We then performed subspace iteration with RR projection using  $k=40$, $m=3$, and $iter=2$ and the Modified Gram-Schmidt QR factorization in Line~6 of Algorithm \ref{alg:multistep subspace iteration}. We examined various precision configurations: \texttt{double} (double-precision floating-point, \texttt{FP64}), \texttt{single} (single-precision floating-point, \texttt{FP32}), and \texttt{half} (half-precision floating-point, \texttt{FP16}). {Note that all computations for \texttt{FP16} are done in \texttt{FP32} and the results are then truncated to \texttt{FP16}. This setting is common since it reflects the behavior of many optimized routines, such as certain \texttt{cuBLAS} functions (like \texttt{cublasDotEx}), which utilize \texttt{FP32} for internal accumulations.} Precision settings for the two operations, MatVecs and QR factorization, are customized for each experiment and designated by labels such as \texttt{[MatVec Precision]-[QR Precision]}. For example, in the \texttt{single-double} configuration, MatVec operations are carried out in single precision while QR factorization is done in double precision. We always use double precision to solve the projected eigenvalue problem in Line~\ref{alg:rayleigh-ritz eigenvalue eigenvalue} of Algorithm \ref{alg:rayleigh-ritz eigenvalue}.

Figure~\ref{fig:test_00} reports the relative errors for the 20 largest eigenvalues under various precision configurations. As anticipated, the configurations with double/single precision MatVecs achieve errors close to the machine epsilon for \texttt{FP64}/\texttt{FP32}, confirming that employing high precision in both MatVecs and QR operations effectively minimizes numerical inaccuracies. 
In contrast, the \texttt{half-half} configurations exhibit substantially larger errors, highlighting the inherent challenges of relying exclusively on half precision.
Comparisons among the \texttt{half-double}, \texttt{half-single}, and \texttt{half-half} configurations suggest that lowering the precision of MatVec operations to half precision alone does not significantly compromise the overall accuracy, as long as the QR factorization is performed in higher precision. 
For example, the approximation accuracy of the \texttt{half-double} configuration surpasses that of full half precision and the error is smaller than $10^{-4}$. 
A similar trend holds for the \texttt{half-single} configuration. 

{Figure \ref{fig:test_00} shows that even when MatVecs are carried out in low precision, the dominant spectral subspace can still be reconstructed accurately once the resulting vectors are post-processed. Because high-precision MatVecs (or the high-precision storage they require) impose heavy penalties in memory traffic and run-time, large-scale solvers already lean toward reduced precision for these operations. The evidence in Figure \ref{fig:test_00} therefore motivates embedding low- or mixed-precision arithmetic not only in the MatVecs but throughout the basis-generation and projection stages. Our goal is to capture the speed and memory advantages of low precision while reserving full precision for the small, projected problem so that the eigenvalue approximation accuracy is not compromised much. The Orthogonalization-Free Rayleigh–Ritz Projection framework, introduced next, is built precisely for this purpose.
}

\begin{figure}[htbp]
     \centering
     \centering
     \includegraphics[width=0.95\textwidth]{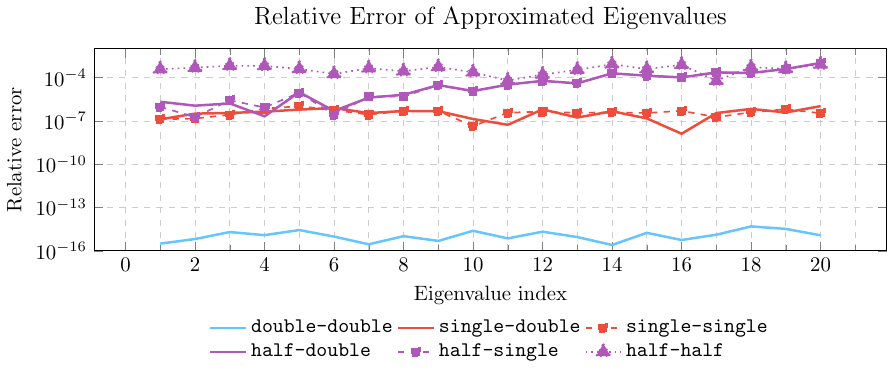}
     \caption{ 
     Relative error plot of subspace iteration with Rayleigh-Ritz projection under different precision options. The test matrix is a Gaussian kernel matrix of size $1000\times 1000$. A concise naming convention is used to denote different options: \texttt{[MatVec Precision]-[QR Precision]}. }
     \label{fig:test_00}
\end{figure}

\section{Orthogonalization-Free Rayleigh-Ritz Projection}\label{sec3}
In this section, we introduce a generalization of the Rayleigh-Ritz projection designed specifically for half-precision and lower. This variant, which does not require an orthogonal input basis, is referred to as the Orthogonalization-Free Rayleigh-Ritz (OFRR) projection method.

\subsection{OFRR for Eigenvalue Problems}\label{sec3-1}

We first consider the eigenvalue problem
\begin{equation}
    \mathbf{A}\mathbf{u} = \lambda \mathbf{u},
\end{equation}
where $\mathbf{A}\in\mathbb{R}^{n\times n}$.
Given a $k$-dimensional subspace $\mathcal{K}$, the orthogonal projection method seeks approximate eigenpairs $\tilde{\lambda}\in \mathbb{C}$, $\tilde{\mathbf{u}} \in\mathcal{K}$ of $\mathbf{A}$ such that the following Galerkin condition is satisfied:
\begin{equation}
\mathbf{A} \tilde{\mathbf{u}} - \lambda	\tilde{\mathbf{u}} \perp \mathcal{K},
\end{equation}
or, equivalently,
\begin{equation}\label{eq:galerkin}
\langle\mathbf{A}\tilde{\mathbf{u}} - \tilde{\lambda}\tilde{\mathbf{u}},\mathbf{u}\rangle = 0, \qquad \forall \mathbf{u} \in \mathcal{K}.
\end{equation}

The standard Rayleigh-Ritz process discussed in Algorithm~\ref{alg:rayleigh-ritz eigenvalue} assumes that an orthonormal basis of $\mathcal{K}$ is available. Now, assume that we only have a linearly independent basis
$\mathbf{U}=[\mathbf{u}_1,\mathbf{u}_2,\cdots,\mathbf{u}_k]$ of $\mathcal{K}$ (which might not be orthogonal). Then the Galerkin condition leads to the following equations:
\begin{equation*}
    \langle\mathbf{A}\tilde{\mathbf{u}} - \tilde{\lambda}\tilde{\mathbf{u}},\mathbf{u}_i\rangle = 0, \qquad i = 1,\cdots, k.
\end{equation*}

Since the approximate solution $\tilde{\mathbf{u}}$ is sought within the subspace $\mathcal{K}$, it can be represented as $\mathbf{U}\mathbf{y}$, with $\mathbf{y}$ being a unique vector in $\mathbb{C}^{k}$. Accordingly, we can transform \eqref{eq:galerkin} into a set of equations involving $\tilde\lambda$ and $\mathbf{y}$:
\begin{equation*}
    \langle\mathbf{A}\mathbf{U}\mathbf{y} - \tilde{\lambda}\mathbf{U}\mathbf{y},\mathbf{u}_i\rangle = 0, \qquad i = 1,\cdots, k.
\end{equation*}
which is equivalent to:
\begin{equation}
\mathbf{U}^{\ast}\mathbf{A}\mathbf{U}\mathbf{y} = \tilde{\lambda}\mathbf{U}^{\ast}\mathbf{U}\mathbf{y}.
\label{eq:geig}
\end{equation} 

This approach, detailed in Algorithm~\ref{alg:of rayleigh-ritz eigenvalue}, shifts from the traditional eigenvalue problem to a generalized one when $\mathbf{U}^{\ast}\mathbf{U} \neq \mathbf{I}$. The effectiveness of this method is closely linked to the condition number of $\mathbf{U}^{\ast}\mathbf{U}$.  We will explore methods for constructing well-conditioned non-orthogonal bases in Section~\ref{sec4}.

\begin{algorithm}[htbp]
\caption{\textit{Orthogonalization-Free Rayleigh-Ritz Projection}}\label{alg:of rayleigh-ritz eigenvalue}
\begin{algorithmic}[1]
\STATE$\triangleright$ \textbf{input:} $\mathbf{A}$, $\mathbf{U}\in \mathbb{C}^{n\times k}$
\STATE$\triangleright$ \textbf{output:} $\tilde{\mathbf{\Lambda}}$ and $\tilde{\mathbf{U}}$ \hfill
\STATE$\triangleright$ Compute $\mathbf{B} = \mathbf{U}^{\ast}\mathbf{A}\mathbf{U}$
\STATE$\triangleright$ Compute $\mathbf{M} = \mathbf{U}^{\ast}\mathbf{U}$
\STATE$\triangleright$ Compute eigendecomposition $\mathbf{B}\mathbf{Y} = \mathbf{M}\mathbf{Y}\tilde{\mathbf{\Lambda}}$
\STATE$\triangleright$ Compute $\tilde{\mathbf{U}} = \mathbf{U}\mathbf{Y}$.
\end{algorithmic}
\end{algorithm}

\subsection{OFRR for Singular Value Decomposition}\label{sec3-2}
In this section, we extend the OFRR method to Singular Value Decomposition (SVD).
We first describe the orthogonal Rayleigh-Ritz projection for SVD.
Consider the following SVD
\begin{equation}\label{eq:svd}
\begin{cases}
    \mathbf{A}\mathbf{v}&=\sigma\mathbf{u} \\
    \mathbf{A}^{\top}\mathbf{u}&=\sigma\mathbf{v}
\end{cases}
\end{equation}
where $\mathbf{A}\in\mathbb{R}^{n_1\times n_2}$, $\mathbf{v}$ is the right singular vector and $\mathbf{u}$ is the left singular vector.
Given two subspaces $\mathcal{K}_1$ of dimension $k_1$ and $\mathcal{K}_2$ of dimension $k_2$, the Rayleigh-Ritz projection for SVD seeks  $\tilde{\sigma}\in\mathbb{R}$, $\tilde{\mathbf{u}}\in\mathcal{K}_1$, and $\tilde{\mathbf{v}}\in\mathcal{K}_2$ such that
\begin{equation}\label{eq:galerkin svd}
\begin{cases}
    \langle\mathbf{A}\tilde{\mathbf{v}} - \tilde{\sigma}\tilde{\mathbf{u}}, \mathbf{u}\rangle = 0, &\qquad \forall \mathbf{u} \in \mathcal{K}_1,\\
    \langle\mathbf{A}^{\top}\tilde{\mathbf{u}} - \tilde{\sigma}\tilde{\mathbf{v}}, \mathbf{v}\rangle = 0, &\qquad \forall \mathbf{v} \in \mathcal{K}_2.
\end{cases}
\end{equation}

When we have an orthonormal basis $\mathbf{U}=[\mathbf{u}_1,\mathbf{u}_2,\ldots,\mathbf{u}_{k_1}]$ for $\mathcal{K}_1$ and an orthonormal basis $\mathbf{V}=[\mathbf{v}_1,\mathbf{v}_2,\ldots,\mathbf{v}_{k_2}]$ for $\mathcal{K}_2$, 
the classical Rayleigh-Ritz projection simply uses the SVD of $\mathbf{U}^{\top}\mathbf{A}\mathbf{V}$ to generate approximate singular values and singular vectors, as shown in Algorithm~\ref{alg:rayleigh-ritz svd}. 

\begin{algorithm}[htbp]
\caption{\textit{Rayleigh-Ritz Projection for SVD}}\label{alg:rayleigh-ritz svd}
\begin{algorithmic}[1]
\STATE$\triangleright$ \textbf{inputs:} $\mathbf{A}\in \mathbb{R}^{{n_1}\times {n_2}}$,  and $\mathbf{U}\in \mathbb{R}^{n_1\times k_1}$, $\mathbf{V}\in \mathbb{R}^{n_2\times k_2}$, both with orthonormal columns.
\STATE$\triangleright$ \textbf{output:} $\tilde{\mathbf{S}}$, $\tilde{\mathbf{U}}$, and $\tilde{\mathbf{V}}$ \hfill\COMMENT{\%Approximate singular values and singular vectors}
\STATE$\triangleright$ Compute $\mathbf{B} = \mathbf{U}^{\top}\mathbf{A}\mathbf{V}$
\STATE$\triangleright$ Compute SVD $\mathbf{B} = \mathbf{Z}\tilde{\mathbf{S}}\mathbf{W}^{\top}$
\STATE$\triangleright$ Compute $\tilde{\mathbf{U}} = \mathbf{U}\mathbf{Z}$, $\tilde{\mathbf{V}} = \mathbf{V}\mathbf{W}$
\end{algorithmic}
\end{algorithm}

Now assume that only linearly independent bases are available for $\mathcal{K}_1$ and $\mathcal{K}_2$ instead of orthonormal ones.
For the Galerkin condition in \eqref{eq:galerkin svd} to be satisfied, we need to impose the following equations:
\begin{equation*}
    \begin{cases}
        \langle\mathbf{A}\tilde{\mathbf{v}} - \tilde{\sigma}\tilde{\mathbf{u}},\mathbf{u}_i\rangle = 0, &\qquad i = 1,\cdots, k_1,\\
        \langle\mathbf{A}^{\top}\tilde{\mathbf{u}} - \tilde{\sigma}\tilde{\mathbf{v}},\mathbf{v}_i\rangle = 0, &\qquad i = 1,\cdots, k_2.
    \end{cases}
\end{equation*}

We can again express any vector $\tilde{\mathbf{u}}\in\mathcal{K}_1$ as $\tilde{\mathbf{u}} = \mathbf{U}\mathbf{y}$ with a unique $\mathbf{y}\in\mathbb{R}^{k_1}$ and $\tilde{\mathbf{v}}\in\mathcal{K}_2$ as $\tilde{\mathbf{v}} = \mathbf{V}\mathbf{z}$ with a unique $\mathbf{z}\in\mathbb{R}^{k_2}$.
We then transform the original problem into the following system of equations in terms of $\tilde{\sigma}$, $\mathbf{y}$, and $\mathbf{z}$
\begin{equation*}
    \begin{cases}
        \langle\mathbf{A}\mathbf{V}\mathbf{z} - \tilde{\sigma}\mathbf{U}\mathbf{y},\mathbf{u}_i\rangle = 0, &\qquad i = 1,\cdots, k_1,\\
        \langle\mathbf{A}^{\top}\mathbf{U}\mathbf{y} - \tilde{\sigma}\mathbf{V}\mathbf{z},\mathbf{v}_i\rangle = 0, &\qquad i = 1,\cdots, k_2,
    \end{cases}
\end{equation*}
which is equivalent to:
\begin{equation}\label{eq:gev for svd}
    \begin{cases}
        \mathbf{U}^{\top}\mathbf{A}\mathbf{V}\mathbf{z} = \tilde{\sigma}\mathbf{U}^{\top}\mathbf{U}\mathbf{y},\\
        (\mathbf{U}^{\top}\mathbf{A}\mathbf{V})^{\top}\mathbf{y} = \tilde{\sigma}\mathbf{V}^{\top}\mathbf{V}\mathbf{z}.
    \end{cases}
\end{equation}

The above system of equations can be reformulated as a block $2$-by-$2$ generalized eigenvalue problem:
\begin{equation}\label{eq:gev for svd matrix}
\begin{bmatrix}
     0 & \mathbf{U}^{\top}\mathbf{A}\mathbf{V} \\ 
    (\mathbf{U}^{\top}\mathbf{A}\mathbf{V})^{\top} & 0
\end{bmatrix}
\begin{bmatrix} \mathbf{y} \\ \mathbf{z}  \end{bmatrix} 
= \tilde \sigma
\begin{bmatrix}
    \mathbf{U}^{\top}\mathbf{U} & 0 \\ 
     0 & \mathbf{V}^{\top}\mathbf{V} 
\end{bmatrix}
\begin{bmatrix} \mathbf{y} \\ \mathbf{z}  \end{bmatrix}.
\end{equation}

When both $\mathbf{U}^{\top}\mathbf{U}$ and $\mathbf{V}^{\top}\mathbf{V}$ are identity matrices, i.e., the bases are orthonormal, this method is equivalent to the orthogonal Rayleigh-Ritz projection described in Algorithm~\ref{alg:rayleigh-ritz svd}.

It is straightforward to see that if $[\mathbf{y}^{\top},\mathbf{z}^{\top}]^{\top}$ is an eigenvector of the generalized eigenvalue problem \eqref{eq:gev for svd matrix} associated with a positive eigenvalue $\tilde \sigma$, then $[-\mathbf{y}^{\top},\mathbf{z}^{\top}]^{\top}$ is an eigenvector associated with $-\tilde \sigma$. Additionally, all other eigenvalues are zero. Therefore, the positive eigenvalues of \eqref{eq:gev for svd matrix} correspond exactly to the singular values that are sought.

The remaining task is to determine the approximate orthonormal singular vectors. The following theorem illustrates how these singular vectors can be constructed from the eigenvectors of \eqref{eq:gev for svd matrix}.

\begin{theorem}
Assume the columns of $[\mathbf{Y}^{\top}, \mathbf{Z}^{\top}]^{\top}$ contain all the eigenvectors associated with the positive eigenvalues of \eqref{eq:gev for svd matrix} and the corresponding eigenvalues are stored in the diagonal matrix $\tilde{\mathbf{S}}$, such that
\begin{equation*}
\begin{bmatrix}
     0 & \mathbf{U}^{\top}\mathbf{A}\mathbf{V} \\ 
    (\mathbf{U}^{\top}\mathbf{A}\mathbf{V})^{\top} & 0
\end{bmatrix}
\begin{bmatrix} \mathbf{Y} \\ \mathbf{Z}  \end{bmatrix} 
= 
\begin{bmatrix}
    \mathbf{U}^{\top}\mathbf{U} & 0 \\ 
     0 & \mathbf{V}^{\top}\mathbf{V} 
\end{bmatrix}
\begin{bmatrix} \mathbf{Y} \\ \mathbf{Z}  \end{bmatrix}\tilde{\mathbf{S}}.
\end{equation*}
Then, the columns of $\tilde{\mathbf{U}} = \sqrt{2}\mathbf{U}\mathbf{Y}$ and $\tilde{\mathbf{V}} = \sqrt{2}\mathbf{V}\mathbf{Z}$ are orthonormal.
\end{theorem}

\begin{proof}
According to the theory of the generalized eigenvalue problem, the eigenvectors of different eigenvalues are orthogonal under the A-inner product defined by the mass matrix. For the two different eigenpairs $(\tilde\sigma_i;[\mathbf{y}_i^{\top},\mathbf{z}_i^{\top}]^{\top})$ and $(\tilde\sigma_j;[\mathbf{y}_j^{\top},\mathbf{z}_j^{\top}]^{\top})$ of \eqref{eq:gev for svd matrix} where $i\neq j$, we have:
\begin{equation}\label{eq:thm eq 1}
\begin{bmatrix} \mathbf{y}_j^{\top} & \mathbf{z}_j^{\top} \end{bmatrix}
\begin{bmatrix} \mathbf{U}^{\top} \mathbf{U} & 0\\ 0 & \mathbf{V}^{\top} \mathbf{V} \end{bmatrix}
\begin{bmatrix} \mathbf{y}_i\\\mathbf{z}_i\end{bmatrix} = 0\Rightarrow \mathbf{y}_j^{\top} \mathbf{U}^{\top} \mathbf{U} \mathbf{y}_i = - \mathbf{z}_j^{\top} \mathbf{V}^{\top} \mathbf{V} \mathbf{z}_i.
\end{equation}

Since $(\tilde\sigma_i;[\mathbf{y}_i^{\top},\mathbf{z}_i^{\top}]^{\top})$ is an eigenpair of \eqref{eq:gev for svd matrix}, we also have
\begin{equation}\label{eq:thm eq 2}
    \mathbf{V}^{\top}\mathbf{A}^{\top}\mathbf{U}\mathbf{y}_i = \tilde{\sigma}_i\mathbf{V}^{\top}\mathbf{V}\mathbf{z}_i\Rightarrow \mathbf{z}_j^{\top}\mathbf{V}^{\top}\mathbf{A}^{\top}\mathbf{U}\mathbf{y}_i = \tilde{\sigma}_i\mathbf{z}_j^{\top}\mathbf{V}^{\top}\mathbf{V}\mathbf{z}_i.
\end{equation}

Similarly, we know that 
\begin{equation}\label{eq:thm eq 3}
    \mathbf{y}_i^{\top}\mathbf{U}^{\top}\mathbf{A}\mathbf{V}\mathbf{z}_j=\tilde{\sigma}_j\mathbf{y}_i^{\top}\mathbf{U}^{\top}\mathbf{U}\mathbf{y}_j = \tilde{\sigma}_j\mathbf{y}_j^{\top}\mathbf{U}^{\top}\mathbf{U}\mathbf{y}_i.
\end{equation}

Given that $\mathbf{z}_j^{\top}\mathbf{V}^{\top}\mathbf{A}^{\top}\mathbf{U}\mathbf{y}_i = \mathbf{y}_i^{\top}\mathbf{U}^{\top}\mathbf{A}\mathbf{V}\mathbf{z}_j$, we can combine \eqref{eq:thm eq 2} and \eqref{eq:thm eq 3} and obtain
\begin{equation}\label{eq:thm eq 4}
    \tilde{\sigma}_i\mathbf{z}_j^{\top}\mathbf{V}^{\top}\mathbf{V}\mathbf{z}_i = \tilde{\sigma}_j\mathbf{y}_j^{\top}\mathbf{U}^{\top}\mathbf{U}\mathbf{y}_i.
\end{equation}
By integrating \eqref{eq:thm eq 1} with \eqref{eq:thm eq 4}, we obtain
\begin{equation*}
    \tilde{\sigma}_i\mathbf{z}_j^{\top}\mathbf{V}^{\top}\mathbf{V}\mathbf{z}_i = \tilde{\sigma}_j\mathbf{y}_j^{\top}\mathbf{U}^{\top}\mathbf{U}\mathbf{y}_i = - \tilde{\sigma}_j\mathbf{z}_j^{\top}\mathbf{V}^{\top}\mathbf{V}\mathbf{z}_i \Rightarrow (\tilde{\sigma}_i+\tilde{\sigma}_j)\mathbf{z}_j^{\top}\mathbf{V}^{\top}\mathbf{V}\mathbf{z}_i = 0,
\end{equation*}
where the second equal sign is due to \eqref{eq:thm eq 1}.
Since both $\tilde{\sigma}_i$ and $\tilde{\sigma}_j$ are positive, $\tilde{\sigma}_i+\tilde{\sigma}_j\neq0$ which implies $\mathbf{z}_j^{\top}\mathbf{V}^{\top}\mathbf{V}\mathbf{z}_i=0$.
We can show a similar property for $\mathbf{y}$ as
\begin{equation}\label{eq:thm conclude 1}
     \mathbf{y}_j^{\top}\mathbf{U}^{\top}\mathbf{U}\mathbf{y}_i = \frac{\tilde{\sigma}_i}{\tilde{\sigma}_j}\mathbf{z}_j^{\top}\mathbf{V}^{\top}\mathbf{V}\mathbf{z}_i = 0.
\end{equation}

Next, we discuss the situation when $i=j$.
In this case, we have: 
\begin{equation}\label{eq:thm eq 5}
\begin{bmatrix} \mathbf{y}_i^{\top} & \mathbf{z}_i^{\top} \end{bmatrix}
\begin{bmatrix} \mathbf{U}^{\top} \mathbf{U} & 0\\ 0 & \mathbf{V}^{\top} \mathbf{V} \end{bmatrix}
\begin{bmatrix} \mathbf{y}_i\\\mathbf{z}_i\end{bmatrix} = 1\Rightarrow \mathbf{y}_i^{\top} \mathbf{U}^{\top} \mathbf{U} \mathbf{y}_i + \mathbf{z}_i^{\top} \mathbf{V}^{\top} \mathbf{V} \mathbf{z}_i = 1.
\end{equation}
Note that obtaining \eqref{eq:thm eq 4} does not require $i\neq j$, so we also have
\begin{equation}\label{eq:thm eq 6}
    \tilde{\sigma}_i\mathbf{z}_i^{\top}\mathbf{V}^{\top}\mathbf{V}\mathbf{z}_i = \tilde{\sigma}_i\mathbf{y}_i^{\top}\mathbf{U}^{\top}\mathbf{U}\mathbf{y}_i\Rightarrow
    \mathbf{z}_i^{\top}\mathbf{V}^{\top}\mathbf{V}\mathbf{z}_i = \mathbf{y}_i^{\top}\mathbf{U}^{\top}\mathbf{U}\mathbf{y}_i.
\end{equation}
Combining \eqref{eq:thm eq 5} and \eqref{eq:thm eq 6}, we have
\begin{equation}\label{eq:thm conclude 2}
    \mathbf{y}_i^{\top} \mathbf{U}^{\top} \mathbf{U} \mathbf{y}_i = \mathbf{z}_i^{\top} \mathbf{V}^{\top} \mathbf{V} \mathbf{z}_i=1/2
\end{equation}

From the results in \eqref{eq:thm conclude 1} and in \eqref{eq:thm conclude 2}, it is obvious that $\tilde{\mathbf{U}} = \sqrt{2}\mathbf{U}\mathbf{Y}$ and $\tilde{\mathbf{V}} = \sqrt{2}\mathbf{V}\mathbf{Z}$ are orthonormal.
\end{proof}

We conclude this section by summarizing the final algorithm in Algorithm \ref{alg:of rayleigh-ritz svd}. The proposed orthogonalization‐free Rayleigh–Ritz projection requires solving a generalized eigenvalue problem of dimension $k_1 + k_2$, in contrast to the standard two‐sided Rayleigh–Ritz projection, which involves an SVD on a $k_1 \times k_2$ matrix. Nonetheless, this approach can preserve good accuracy even under loss of orthogonality in low‐precision computations.

\begin{algorithm}[htbp]
\caption{\textit{Orthogonalization-Free Rayleigh-Ritz Projection for SVD}}\label{alg:of rayleigh-ritz svd}
\begin{algorithmic}[1]
\STATE$\triangleright$ \textbf{input:} $\mathbf{A}\in \mathbb{R}^{{n_1}\times {n_2}}$, $\mathbf{U}\in \mathbb{R}^{n_1\times k_1}$, $\mathbf{V}\in \mathbb{R}^{n_2\times k_2}$
\STATE$\triangleright$ \textbf{output:} $\tilde{\mathbf{S}}$, $\tilde{\mathbf{U}}$, and $\tilde{\mathbf{V}}$ \hfill\COMMENT{\%Approximate singular values and singular vectors}
\STATE$\triangleright$ Solve the following generalized eigenvalue problem for all positive eigenvalues
\begin{equation*}
\begin{bmatrix}
     0 & \mathbf{U}^{\top}\mathbf{A}\mathbf{V} \\ 
    (\mathbf{U}^{\top}\mathbf{A}\mathbf{V})^{\top} & 0
\end{bmatrix}
\begin{bmatrix} \mathbf{y} \\ \mathbf{z}  \end{bmatrix} 
= 
\tilde{\sigma}\begin{bmatrix}
    \mathbf{U}^{\top}\mathbf{U} & 0 \\ 
     0 & \mathbf{V}^{\top}\mathbf{V} 
\end{bmatrix}
\begin{bmatrix} \mathbf{y} \\ \mathbf{z}  \end{bmatrix}
\end{equation*}
\STATE$\triangleright$ Assemble the eigenvectors associated with positive eigenvalues into the columns of matrices $\mathbf{Y}$ and $\mathbf{Z}$, and insert these eigenvalues into the diagonal of the diagonal matrix $\tilde{\mathbf{S}}$.
\STATE$\triangleright$  Compute $\tilde{\mathbf{U}} = \sqrt{2}\mathbf{U}\mathbf{Y}$, $\tilde{\mathbf{V}} = \sqrt{2}\mathbf{V}\mathbf{Z}$
\end{algorithmic}
\end{algorithm}



%
%
{
\section{Construction of Linearly Independent Bases}\label{sec4}

In the previous section, we proposed the OFRR method for eigenvalue problems and SVD. Unlike orthogonal projection methods, which require an orthogonal basis, OFRR allows for more flexibility with non-orthogonal bases. This section will focus on various strategies to enhance the linear independence of bases for effective integration with OFRR. The procedure for SVD is similar to the eigenvalue computations, with the only difference being the additional matrix-vector multiplications with $\mathbf{A}^{\top}$. For the sake of conciseness, we omit the discussion of the SVD algorithm.

\subsection{Linearly Independent Basis for Krylov Subspace Methods}\label{sec:krylov}
In this section, we will focus on generating linearly independent bases for the Krylov subspace $\mathcal{K}_{k}(\mathbf{v},\mathbf{A})$ for $\mathbf{A}\in\mathbb{R}^{n\times n}$.
Under the OFRR framework, where orthogonality is not required, the simplest approach is to directly use the matrix $\mathbf{K}:=[\mathbf{v},\mathbf{A}\mathbf{v},\cdots,\mathbf{A}^{k-1}\mathbf{v}]$ without any modification.
The generalized eigenvalue problem using OFRR would then be:
\begin{equation}
    \mathbf{K}^{\top}\mathbf{A}\mathbf{K}\mathbf{y}=\tilde{\lambda}\mathbf{K}^{\top}\mathbf{K}\mathbf{y}.
\end{equation}
However, there are several reasons why this approach is generally not recommended.

First, some columns of $\mathbf{K}$ might be nearly linearly dependent, especially when the original matrix $\mathbf{A}$ is numerically low-rank. Direct use of $\mathbf{K}$ could result in a mass matrix $\mathbf{K}^{\top}\mathbf{K}$ that is extremely ill-conditioned in this case, which adversely affects the numerical stability of the eigenvalue algorithm. Second, overflow can arise in computations especially with  reduced precision. While column scaling might be applied to normalize the infinity norm of each column of $\mathbf{K}$ to one, the magnitudes of the columns' $2$-norms can remain large. This can potentially lead to overflow when forming $\mathbf{K}^{\top}\mathbf{A}\mathbf{K}$ and $\mathbf{K}^{\top}\mathbf{K}$ in reduced-precision environments.
Therefore, it is still essential to use algorithms that avoid poorly conditioned bases.

\subsubsection{Arnoldi Method}

A straightforward approach is to employ standard methods for constructing an orthogonal basis, simply executing them using reduced precision arithmetic. For instance, the Arnoldi method – the most widely adopted technique for building an orthogonal basis of the Krylov subspace associated with a general matrix – could be applied to construct linearly independent bases.
One variant of the Arnoldi algorithm is shown in Algorithm~\ref{alg:Krylov arnoldi processs}, where Modified Gram-Schmidt (MGS) is used to build an orthogonal basis for the Krylov subspace.
{It is worth noting that when the input matrix $\mathbf{A}$ is symmetric, applying this general Arnoldi procedure becomes computationally equivalent to the Lanczos algorithm with full orthogonalization, a variant often employed for enhanced numerical stability.}
In some applications implemented using double precision, where orthogonality is critical, re-orthogonalization is typically enabled.
With re-orthogonalization, Lines~\ref{alg:Krylov arnoldi processs mgs start}--\ref{alg:Krylov arnoldi processs mgs end} in Algorithm~\ref{alg:Krylov arnoldi processs} are repeated once if the $2$-norm of $\mathbf{v}$ after projection is reduced by more than a certain tolerance.

\begin{algorithm}[htbp]
\caption{\textit{Computing orthogonal bases from the Arnoldi Process with MGS}}\label{alg:Krylov arnoldi processs}
\begin{algorithmic}[1]
\STATE \textbf{input:} $\mathbf{A}\in\mathbb{C}^{n\times n}$, $\mathbf{v}$, $k$
\STATE \textbf{output:} $\mathbf{V}$ \hfill\COMMENT{linearly independent basis for $\mathcal{K}_{k}(\mathbf{A},\mathbf{v})$}
\STATE$\triangleright$ Initialize matrix $\mathbf{V}$
\STATE$\triangleright$ Update $\mathbf{v}:= \mathbf{v}/\|\mathbf{v}\|_2$
\STATE$\triangleright$ Set $\mathbf{V}_{:,1} = \mathbf{v}$
\FOR {$j=1:k-1$}
    \STATE$\triangleright$ Compute $\mathbf{v} = \mathbf{A} \mathbf{V}_{:,j}$
    \FOR {$i=1:j$} \label{alg:Krylov arnoldi processs mgs start}
        \STATE$\triangleright$ Compute $\mathbf{v} := \mathbf{v} - \langle \mathbf{V}_{:,i}, \mathbf{v} \rangle \mathbf{V}_{:,i} $
    \ENDFOR \label{alg:Krylov arnoldi processs mgs end}
    \STATE$\triangleright$ Update $\mathbf{v}:= \mathbf{v}/\|\mathbf{v}\|_2$
    \STATE$\triangleright$ Set $\mathbf{V}_{:,j+1} = \mathbf{v}$
\ENDFOR
\end{algorithmic}
\end{algorithm}

Classical Gram-Schmidt (CGS) with re-orthogonalization is also commonly used in scientific computing because it can leverage \texttt{BLAS}  level-2 operations for computational performance and significantly reduces the number of parallel reduction operations (required in computing inner products) compared to MGS.
Note that in the context of OFRR, re-orthogonalization may  not be needed.

For $\mathbf{A}\in\mathbb{R}^{n\times n}$, computing the $j$-th column of $\mathbf{V}$ requires approximately $4nj$ FLOPs, leading to a total cost of roughly $2nk^2$ excluding the matrix-vector multiplication with $\mathbf{A}$.
While CGS has the same approximate FLOP count ($2nk^2$), it differs structurally from MGS by using \texttt{BLAS} level-2 operations. 
Specifically, it employs matrix-vector products to compute sets of inner products $\langle \mathbf{V}_{:,i}, \mathbf{v} \rangle$ and perform vector updates, contrasting with the \texttt{BLAS} level-1 operations used in MGS.
Re-orthogonalization would double the cost to $4nk^2$ for both methods. Alternatively, Householder reflectors can also be used to generate the orthonormal basis, offering superior numerical stability.
However, explicit formation of $\mathbf{V}$, necessary for certain applications, makes the total FLOP count approximately $4nk^2$.

Therefore, in reduced-precision environments with OFRR where strict orthonormality may not be required, CGS or MGS without re-orthogonalization offer improved FLOP efficiency and are often preferred to Householder reflectors. 
In the following sections, we will only discuss the use of CGS and MGS.

While the Arnoldi process using Gram-Schmidt orthogonalization can provide good numerical stability for OFRR, its reliance on full orthogonalization is often computationally expensive. A primary reason for this expense is the frequent requirement for inner product computations inherent in Gram-Schmidt methods.

Furthermore, performing these inner products in low-precision formats, such as half precision, presents significant challenges beyond just the computational cost. The limited dynamic range increases the risk of overflow or underflow during summation, and precision loss can severely compromise the numerical stability of the orthogonalization process. While strategies like accumulating inner products in higher precision or applying dynamic vector scaling can mitigate these issues, they introduce additional computational overhead.

To reduce computational demands, we will explore alternative methods or modifications that mitigate the cost and numerical issues arising from inner product computations in reduced precision in the next section.

\subsubsection{Krylov-Hessenberg Process}
In this section, we propose to adopt the Hessenberg process as an alternative to generate linearly independent bases.
This method is derived from the Generalized Hessenberg process, as detailed in Wilkinson’s classical book \cite[Chap. 6]{wilkinson1988algebraic}. Unlike traditional methods that depend on inner products to compute the projection coefficients, the Hessenberg process obtains projection coefficients by extracting entries directly from previously computed bases.
}
More specifically, the procedure generates the basis vectors  $\mathbf{v}_1, \mathbf{v}_2, \cdots, \mathbf{v}_k$ of the Krylov subspace in the usual Arnoldi-like fashion except that orthogonality is enforced against a preselected set of vectors $\mathbf{z}_1, \mathbf{z}_2, \cdots, \mathbf{z}_k$ instead of the $\mathbf{v}_i$'s themselves. 
Thus, at step $j$ of the procedure we compute the vector $\mathbf{A} \mathbf{v}_j$ and orthogonalize it against $\mathbf{z}_1, \mathbf{z}_2, \cdots, \mathbf{z}_j$, leading to a vector $\mathbf{v}_{j+1}$ that satisfies the usual relation among Krylov basis vectors:
\begin{equation}\label{eq:KryHess1}
\mathbf{h}_{j+1,j} \mathbf{v}_{j+1} = \mathbf{A} \mathbf{v}_j -\sum_{i=1}^j \mathbf{h}_{i,j} \mathbf{v}_i .
\end{equation}

{In this paper, we consider the simplest case where we choose $\mathbf{z}_i=\mathbf{e}_i$, and scale all $\mathbf{v}_i$'s so that $\|\mathbf{v}_i\|_\infty=1$.}
This procedure has been advocated in \cite{sadok1999cmrh} as an alternative to GMRES for solving linear systems of equations iteratively.
Later it was also exploited for solving \emph{dense} linear systems, see, for example, \cite{heyouni2008new}.  
The Hessenberg algorithm for generating a non-orthogonal basis is summarized in  Algorithm~\ref{alg:Krylov hessenberg processs}.

\begin{algorithm}[htbp]
\caption{\textit{Computing non-orthogonal bases from the Krylov-Hessenberg Process}}\label{alg:Krylov hessenberg processs}
\begin{algorithmic}[1]
\STATE \textbf{input:} $\mathbf{A}\in\mathbb{C}^{n\times n}$, $\mathbf{v}$, $k$
\STATE \textbf{output:} $\mathbf{V}$ \hfill\COMMENT{linearly independent basis for $\mathcal{K}_{k}(\mathbf{A},\mathbf{v})$}
\STATE$\triangleright$ Initialize matrix $\mathbf{V}$
\STATE$\triangleright$ Initialize permutation vector $\mathbf{\pi}$
\STATE$\triangleright$ Find $r$ the index of element in $\mathbf{v}$ with largest magnitude \label{alg:Krylov hessenberg processs max1}
\STATE$\triangleright$ Update $\mathbf{v}:= \mathbf{v}/\mathbf{v}_{r}$
\STATE$\triangleright$ Set $\mathbf{\pi}_1:= r$
\STATE$\triangleright$ Set $\mathbf{V}_{:,1} = \mathbf{v}$
\FOR {$j=1:k-1$}
    \STATE$\triangleright$ Compute $\mathbf{v} = \mathbf{A} \mathbf{V}_{:,j}$
    \FOR {$i=1:j$}
        \STATE$\triangleright$ Compute $\mathbf{v} := \mathbf{v} - \mathbf{v} (\mathbf{\pi}_i)  \mathbf{V}_{:,i} $
    \ENDFOR
    \STATE$\triangleright$ Find $r$ the index of element in $\mathbf{v}$ with largest magnitude. \label{alg:Krylov hessenberg processs max2}
    \STATE$\triangleright$ Update $\mathbf{v}:= \mathbf{v}/\mathbf{v}_{r}$
    \STATE$\triangleright$ Set $\mathbf{\pi}_{j+1}:= r$
    \STATE$\triangleright$ Set $\mathbf{V}_{:,j+1} = \mathbf{v}$
\ENDFOR
\end{algorithmic}
\end{algorithm}

{As can be seen the Hessenberg process is inner-product free. Only one reduction operator to obtain the index of the entry with the largest magnitude is needed during each outer step (Line~\ref{alg:Krylov hessenberg processs max2} of Algorithm \ref{alg:Krylov arnoldi processs}). This constitutes a significant advantage over the Arnoldi process. 
The arithmetic operations involved are also less prone to numerical stability issues as will be seen later.
Note also that he FLOP count for the Hessenberg process {when $n\gg k$ is roughly $nk^2$} excluding the matrix-vector multiplication with $\mathbf{A}$.}


\subsection{Linearly Independent Basis for Subspace Iteration}

{
The previous discussions have focused on strategies for constructing linearly independent bases for the Krylov subspace. We now shift focus to subspace iteration.

Subspace iteration, in contrast, adopts a block-oriented approach. It allows for the potential use of higher-level \texttt{BLAS} operations (e.g., \texttt{BLAS} level-3 for the matrix-block product if $\mathbf{A}$ is dense) and can lead to improved computational efficiency on modern architectures compared to the Krylov methods.
Similar to the challenges encountered in naive Krylov-based implementations, directly applying $\mathbf{X}_{\text{iter}} = \mathbf{A}^{\text{iter}} \mathbf{X}_0$ without modification may result in severe numerical difficulties.

\subsubsection{QR factorization}

A straightforward way to construct a linearly independent basis with an improved condition number is to run a QR factorization with column pivoting using either CGS or MGS. 
Although the arithmetic work is identical to that in Arnoldi—$2mn^{2}$ FLOPs for a single sweep and $4mn^{2}$ FLOPs with re‑orthogonalization—the practical speed can differ greatly. 

\begin{figure}[htbp]
    \centering
    \includegraphics[width=0.25\linewidth]{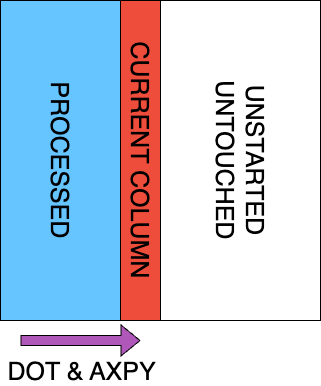}
    \hspace{50pt}
    \includegraphics[width=0.25\linewidth]{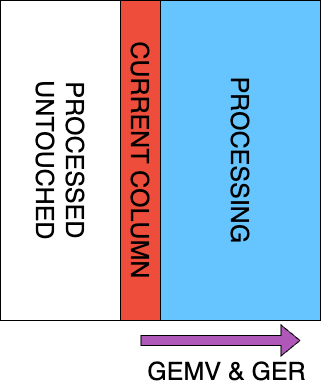}
    \caption{The MGS sweep used within Arnoldi (left) and a ``right‑looking'' variant of MGS for subspace iteration (right). In Arnoldi each processed basis vector updates the single current column via repeated \texttt{DOT} \& \texttt{AXPY} operations. In the `righ-looking' version, once the current column is orthogonal, a single \texttt{GEMV} \& \texttt{GER} applies its correction to the entire trailing block at once.}
    \label{fig:mgs versions}
\end{figure}

As shown in the left panel of Figure~\ref{fig:mgs versions}, the MGS sweep embedded in Arnoldi is ``left‑looking'': at step $j$ only the columns already computed are available, so the projection must be carried out through $j$ successive \texttt{DOT}–\texttt{AXPY} pairs (\texttt{BLAS} level-1), making the computation memory‑bound. 
In contrast, for subspace iteration the entire 
block is resident in memory; the sweep can therefore be organized in a ``right‑looking'' manner (right panel of Figure~\ref{fig:mgs versions}), where one \texttt{GEMV} followed by a rank‑1 \texttt{GER} updates all trailing columns at once. Packaging the same FLOPs into \texttt{BLAS} level-2 calls raises arithmetic intensity and yields markedly higher sustained performance on modern CPUs and GPUs. 

\subsubsection{The Hessenberg Process}
The Krylov-Hessenberg process introduced in the previous section can be readily modified for building a linearly independent basis for the subspace iteration.
Similar to MGS, the Hessenberg process for subspace iteration could also be implemented in a  `right-looking' way, as detailed in Lines 13-15 of Algorithm~\ref{alg:hessenberg process}. 
The output of this algorithm returns a linearly independent basis $\mathbf{Q}=[\mathbf{q}_1,\mathbf{q}_2,\ldots]$. Since the columns of $\mathbf{A}$ may be nearly linearly dependent, it is crucial to skip columns with negligible magnitude during factorization as implemented in Line 9 of Algorithm~\ref{alg:hessenberg process}. Specifically, when selecting row pivots, if the largest magnitude entry in a column is close to the working precision, this column should be skipped, and the factorization should continue with the next column. The objective is to ensure that $\spann(\mathbf{Q})$ closely approximates $\spann(\mathbf{A})$. Therefore, zero or near-zero columns should be omitted rather than padded with standard basis vectors to reflect 
the correct numerical rank.
 
\begin{algorithm}[htbp]
\caption{\textit{Computing non-orthogonal bases from the Hessenberg Process}}\label{alg:hessenberg process}
\begin{algorithmic}[1]
\STATE \textbf{input:} $\mathbf{A}\in\mathbb{C}^{n\times k}$, $tol$\hfill\COMMENT{$tol$ is to exclude zero columns}\\
\STATE \textbf{output:} $\mathbf{Q}$ \hfill\COMMENT{linearly independent basis}
\STATE$\triangleright$ Initialize matrix $\mathbf{Q}=\mathbf{A}$
\STATE$\triangleright$ Initialize nonzero column indicator vector $\mathbf{s}$ to \texttt{TRUE}
\STATE$\triangleright$ Initialize permutation vector $\mathbf{\pi}$
\FOR {$j=1:k$} 
\STATE$\triangleright$ Set $\mathbf{q} = \mathbf{Q}_{:.j}$
    \STATE$\triangleright$ Find $r$ the index of element in $\mathbf{q}$ with largest magnitude.
    \IF{$\vert\mathbf{q}_k\vert\geq tol$}
    \STATE$\triangleright$ Update $\mathbf{q} := \mathbf{q}/\mathbf{q}_r$ 
    \STATE$\triangleright$ Set $\mathbf{\pi}_j = r$
    \STATE$\triangleright$ Set $\mathbf{Q}_{:,j} = \mathbf{q}$
    \FOR {$i = j+1:k$}\label{alg:hessenberg process inner loop start}
        \STATE$\triangleright$ Update $\mathbf{Q}_{:,i}:= \mathbf{Q}_{:,i}-\mathbf{Q}_{\mathbf{\pi}_i,i}\mathbf{q}$
    \ENDFOR\label{alg:hessenberg process inner loop end}
    \ELSE
    \STATE$\triangleright$ Set $\mathbf{s}_k$ to \texttt{FALSE}
    \STATE$\triangleright$ Set $\mathbf{\pi}_r=1$ 
    \ENDIF
\ENDFOR
\STATE$\triangleright$ Set $\mathbf{Q}=\mathbf{Q}_{:,\mathbf{s}}$\label{alg:hessenberg process - line selection}
\end{algorithmic}
\end{algorithm}

Finally, we discuss the connection between the Hessenberg process and the LU factorization.
For a given $\mathbf{A}\in\mathbb{C}^{n\times k}$ with full column rank, LU factorization with row pivoting computes
\begin{equation}
    \mathbf{P}\mathbf{A} = \mathbf{L}\mathbf{U}\Rightarrow \mathbf{A} = (\mathbf{P}^{\top}\mathbf{L})\mathbf{U},
\end{equation}
where $\mathbf{L}\in\mathbb{C}^{n\times k}$, $\mathbf{U}\in\mathbb{C}^{k\times k}$, and $\mathbf{P}\in\mathbb{R}^{n\times n}$ is a permutation matrix, i.e., a matrix obtained by reordering the rows of an identity matrix. 
The matrix $\mathbf{P}^{\top}\mathbf{L}$ now has the same range as $\mathbf{A}$, and its columns could be used as a linearly independent basis for the column space of $\mathbf{A}$. 

Although both the Hessenberg process and the LU factorization have been widely used, their direct algorithmic relationship is worth highlighting.
An interesting observation is that the output matrix produced by the Hessenberg process in Algorithm~\ref{alg:hessenberg process} is identical to the output matrix obtained from the row-pivoted LU factorization.
Thus, the numerical stability analysis of the Hessenberg process is supported by existing results on the LU factorization. 

Motivated by this connection, we next examine advances in mixed-precision LU factorization, which has become a topic of significant interest due to its potential for accelerating the solution of large-scale linear systems of the form $\mathbf{A}\mathbf{x} = \mathbf{b}$. One development was made by Haidar et al.~\cite{haidar2018harnessing}, who proposed a low-precision LU factorization strategy 
 within an iterative refinement framework. Their work relied on a partitioned, `right-looking' LU algorithm designed to maximize data locality and arithmetic intensity. 
More recently, Lopez et al. \cite{lopez2023mixed} proposed a highly efficient mixed-precision approach based on a partitioned `left-looking' LU factorization.
The core idea behind such partitioned approaches is to perform the factorization on $r\times r$ blocks, allowing the algorithms to leverage high-performance \texttt{BLAS} level-3 operations.
Another contribution  \cite{sahraneshinsamani2025mixed}, explores pre-pivoting strategies within mixed-precision frameworks, where the ultimate goal is to achieve effective \texttt{FP64} accuracy.

Despite the focus of these efforts on solving linear systems, rather than subspace generation, the methodological innovations developed therein are highly relevant to our proposed work. Specifically, the structured block-wise execution and memory-aware optimizations introduced in partitioned LU schemes may be adapted to enhance the efficiency of the Hessenberg process when used for constructing basis matrices in Krylov subspace and subspace iteration methods. Such a generalization would not only expand the utility of the Hessenberg process but could also yield significant performance benefits on modern computing architectures. A rigorous exploration of these extensions—especially in the context of block Krylov methods—presents a promising direction for future research.

}


\subsubsection{Condition Number Comparison of Computed Bases}
In this subsection, we compare the condition numbers of the bases generated by various methods discussed in the previous subsections. We choose Gaussian kernel matrices with data uniformly sampled within a square of edge length $\sqrt{1000}$ in the experiment. This time, we fix $s = 0.01$ and vary $l$ from $1$ to $100$ in \eqref{eq:guassiankernel} to test matrices with different spectral properties. When $l$ is close to 1, the eigenvalues of the matrix decay slowly, and as $l$ approaches 100, most of its eigenvalues would be close to $s$.
To be specific, when $l=1$ the $20$-th largest eigenvalue is larger than $6$, while when $l=100$ the $7$-th largest eigenvalue is already close to $0.01$. We perform subspace iteration with $m=1$, $\text{iter}=3$, and $k=20$ to compute $\mathbf{X}_{iter}$, and then compute the condition number of the postprocessed basis matrix \(\mathbf{Q}\). Specifically, we calculate the condition numbers of \(\mathbf{X}_{iter}\), \(\mathbf{Q}\) obtained through Modified Gram-Schmidt (MGS) with re-orthogonalization, Classical Gram-Schmidt (CGS) with re-orthogonalization, and the one returned from the Hessenberg process. 
The subspace iteration and basis computations are performed in double, single, and half precision. {As the results in single precision are consistently close to those in double precision, we exclude them from the figure to improve readability.} {Notably, for the half-precision configuration in this experiment, the computations are carried out using native  \texttt{FP16} arithmetic without intermediate calculations in \texttt{FP32}, presenting a more challenging scenario for maintaining numerical stability.} The final condition numbers are calculated in double precision for accuracy.

\begin{figure}[htbp]
    \centering
    \includegraphics[width=0.95\textwidth]{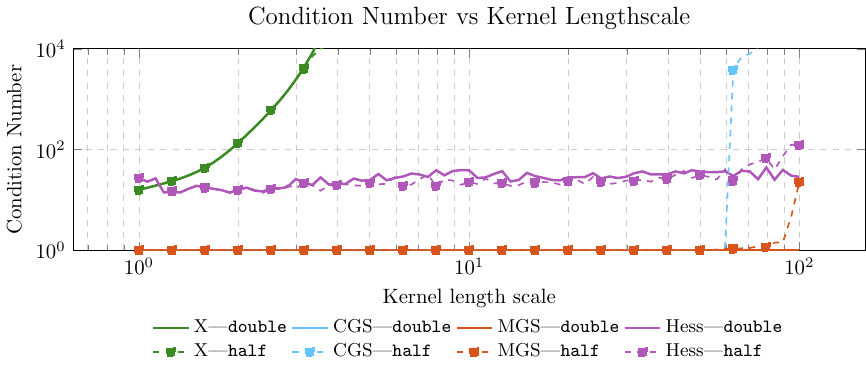}
    \caption{Condition number for bases computed by four different methods: no stabilization (-X), MGS with re-orthogonalization, CGS with re-orthogonalization, and Hessenberg. Tests are performed on multiple kernel matrices, each sized $1000\times 1000$, with length scales varying from $1$ to $100$.}
    \label{fig:condition number of basis}
\end{figure}

As we can see from Figure~\ref{fig:condition number of basis}, the condition number of the original $\mathbf{X}_{iter}$ increases significantly as the problem becomes numerically low-rank, necessitating the use of QR-like strategies for building a well-conditioned basis. On the other hand, double-precision MGS and CGS (with re-orthogonalization) consistently produce a basis with a condition number close to 1.
However, as the matrices become numerically low-rank, MGS and CGS begins to lose orthogonality under half-precision, leading to higher condition numbers. 
Although the condition numbers for the Hessenberg process are slightly larger than those for MGS and CGS in most cases, they exhibit consistently low variance across different length scales and precisions. This stability, combined with its efficiency, makes it a robust and efficient option for building the bases with reduced precision arithmetic.

Next, Table~\ref{tab:method_comparison} provides a comparative summary of the dominant \texttt{BLAS} operations, FLOPs, parallelism, and numerical stability across different variants of Modified Gram-Schmidt (MGS), Classical Gram-Schmidt (CGS), and the Hessenberg process.

\begin{table}[htbp]
\centering
\caption{Comparison of MSG, CGS and Hessenberg Process when applied to an $n\times k$ matrix $\mathbf{A}$ {when $n\gg k$}.}
\label{tab:method_comparison}
\resizebox{\textwidth}{!}{
\begin{tabular}{@{}lllll@{}}
\toprule
Method Variant & Dominant \texttt{BLAS} & FLOPs (Order) & Parallelism & Stability  \\
\midrule

\multicolumn{5}{c}{\textbf{Modified Gram-Schmidt (MGS)}} \\
\midrule
Left-Looking MGS & Level 1 & $2nk^2$ & Low & Moderate  \\
 \addlinespace
Left-Looking MGS re-orth & Level 1 & $4nk^2$ & Low & Good  \\
 \addlinespace
Right-Looking MGS & Level 2 & $2nk^2$ & Moderate & Moderate \\
\midrule

\multicolumn{5}{c}{\textbf{Classical Gram-Schmidt (CGS)}} \\
\midrule
CGS & Level 2 & $2nk^2$ & Moderate & Low \\
 \addlinespace
CGS2 & Level 2 & $4nk^2$ & Moderate & Moderate \\
\midrule

\multicolumn{5}{c}{\textbf{Hessenberg Process}} \\
\midrule
Left-Looking & Level 1 & $nk^2$ & Moderate & Good \\
 \addlinespace
Right-Looking & Level 2 & $nk^2$ & Moderate & Good \\
 \addlinespace
Block & Level 3 & $nk^2$ & Good & Good \\
\bottomrule
\end{tabular}%
}
\end{table}

{
To conclude this section, we outline the mixed-precision strategies employed within the OFRR framework. These strategies are adapted  based on the precision in which the input matrix $\mathbf{A}$ is available. Figure~\ref{fig:diagram} serves to illustrate the data flow and suggested precision choices for one important scenario: applying subspace iteration to a matrix $\mathbf{A}$ provided only in half precision (\texttt{FP16}). When employing Krylov subspace methods within the OFRR framework, the core principle of major steps remains similar.

As illustrated in Figure~\ref{fig:diagram}, memory conservation is prioritized by storing the basis vectors $\mathbf{V}$ and intermediate results like $\mathbf{W} :=\mathbf{A}\mathbf{V}$ primarily in \texttt{FP16} format. Computations such as applying matrix-matrix multiplication with $\mathbf{A}$ and performing the basis construction (OFRR Hessenberg/QR) can often use \texttt{FP16} compute precision, potentially enhanced with \texttt{FP32} accumulation. A critical step is the projection step to form the matrices $\mathbf{B}:=\mathbf{V}^\top \mathbf{W}$ and $\mathbf{M}:=\mathbf{V}^\top \mathbf{V}$. This step takes input matrices (like $\mathbf{V}$ and $\mathbf{W}$) in \texttt{FP16}, performs the matrix multiplications and accumulations using \texttt{FP32} compute precision, and stores the resulting small matrices $\mathbf{B}$ and $\mathbf{M}$ in \texttt{FP32}. Finally, these \texttt{FP32} matrices are promoted to \texttt{FP64} to solve the generalized eigenvalue problem defined by matrix pencil $(\mathbf{B}, \mathbf{M})$ using standard double-precision solvers.

The strategy is much simpler for subspace iteration with higher-precision inputs. If $\mathbf{A}$ is \texttt{FP64}, all storage and computations throughout the process typically remain in \texttt{FP64}. If $\mathbf{A}$ is \texttt{FP32}, the framework generally operates with \texttt{FP32} as the working precision for both storage and computation, with the final projected generalized eigenvalue problem solved in \texttt{FP64} to maximize the accuracy of the resulting eigenpairs.

In summary, the OFRR framework flexibly integrates mixed-precision strategies, leveraging low-precision storage and computation where feasible, while strategically increasing precision for numerically sensitive stages like projection and the final solve.

}

\begin{figure}[htbp]
    \centering
    \includegraphics[width=0.995\linewidth]{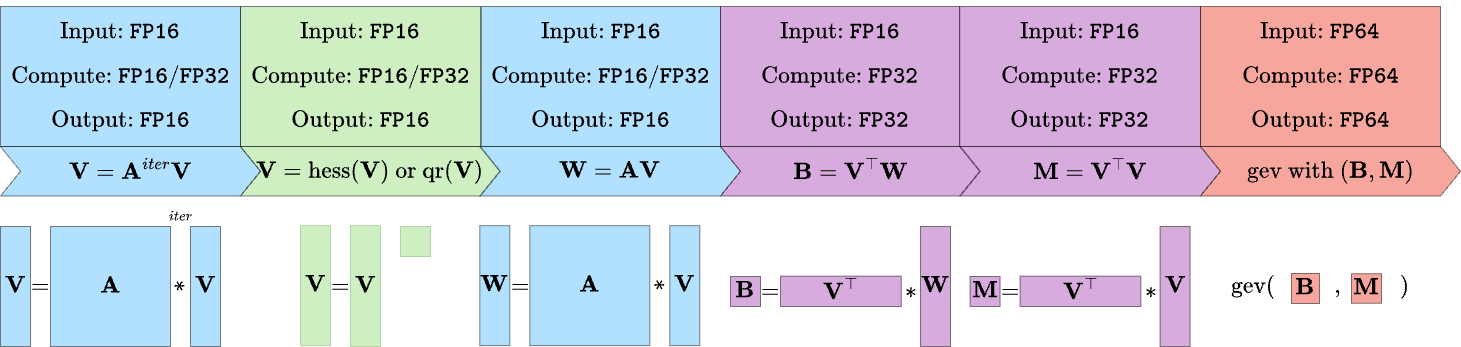}
    \caption{Mixed-precision strategy within the OFRR framework using subspace iteration for an \texttt{FP16} input matrix $\mathbf{A}$. Each stage indicates typical precision choices for data storage (Input/Output) and computation (Compute), illustrating the progression from \texttt{FP16} working precision to \texttt{FP32} projection and an \texttt{FP64} final solve.}
    \label{fig:diagram}
\end{figure}

%
%



%
%

\section{Numerical Experiments}\label{sec5}

In this section, we test our OFRR framework on different problems that require (partial) eigenvalue  or singular value decompositions under various numerical precision settings.
Our evaluation consists of two parts: (i) numerical accuracy is assessed using simulations implemented in \texttt{MATLAB} (version 2024b); (ii) computational efficiency is evaluated using an optimized \texttt{C++} implementation accelerated by CUDA, compiled with \texttt{nvcc} (version 12.8).
All experiments were conducted on a hardware platform running Ubuntu 24.04.2 LTS, equipped with an Intel(R) Core(TM) i7-12700K CPU (8 Performance-cores @ 3.60 GHz and 4 Efficient-cores @ 2.70 GHz), 64 GB of system memory, and a \texttt{NVIDIA} GeForce RTX 3070 Ti GPU (8 GB GDDR6X VRAM, compute capability 8.6, and 6144 \texttt{CUDA} cores).
For reproducibility, our research code is publicly available on \texttt{GitHub}\footnote{\href{https://github.com/Hitenze/MixedPrecisionOFRR}https://github.com/Hitenze/MixedPrecisionOFRR}.

\subsection{Implementation Details}

Our \texttt{MATLAB} implementation was designed to rigorously evaluate the numerical accuracy of the proposed OFRR framework across different precision settings.

We simulated half-precision arithmetic using \texttt{MATLAB}'s built-in \texttt{half} datatype and added custom functions to precisely control numerical precision. A key detail is that \texttt{MATLAB}'s standard operations on \texttt{half} arrays often perform intermediate computations in higher precision, and only convert the final result back to \texttt{FP16}. This type of mixed-precision behavior is common in libraries like \texttt{cuBLAS}, although some GPU routines support full \texttt{FP16} execution.

To ensure that all computations adhered strictly to our intended precision model, we required full control over every arithmetic step. For this reason, we avoided using \texttt{MATLAB}'s built-in \texttt{qr} function and instead implemented our own versions of the MGS algorithm with re-orthogonalization to serve as the QR-based baseline. This approach allowed us to guarantee that every operation followed the specified precision path, with no hidden accuracy promotions or conversions. For MGS, we used the standard threshold $\sqrt{2}/2$ to detect loss of orthogonality and trigger re-orthogonalization.

Furthermore, several other custom functions were necessary because \texttt{MATLAB} lacks native half-precision support for certain operations.
We implemented a custom $2$-norm function specifically for low precision, using the standard technique of scaling the vector by its infinity norm before computing the $2$-norm, i.e., $\|\mathbf{x}\|_2=\|\mathbf{x}\|_{\infty}\|\mathbf{x}/\|\mathbf{x}\|_{\infty}\|_2$.
This technique mitigates overflow and underflow issues that are common in low-precision arithmetic.
For sparse matrix operations, we implemented custom routines based on the Compressed Sparse Row (CSR) format, chosen for its implementation simplicity.
For experiments conducted in single precision and double precision (FP64), we used standard \texttt{MATLAB} data types and built-in functions.

In addition to the \texttt{MATLAB} implementation for accuracy studies, we developed \texttt{C++}/\texttt{CUDA} implementations of key subroutines to evaluate runtime performance on GPUs. 
The \texttt{C++} implementation employs the standard \texttt{FP16} data type defined in \texttt{cuda\_fp16.h}, and integrates functions from \texttt{cuBLAS}, standard \texttt{BLAS}, and \texttt{LAPACK}.
All dense matrices are stored in column-major order (\texttt{Fortran}-style), and all \texttt{cuBLAS} calls and custom kernels are executed on the default \texttt{CUDA} stream.
Unless otherwise specified, the primary matrix data resides in device memory during computation.
Scalar parameters (e.g., weights for linear combinations or column scaling factors) used in \texttt{cuBLAS} routines are passed from host memory by configuring the \texttt{cuBLAS} pointer mode \texttt{CUBLAS\_POINTER\_MODE\_HOST}.

For `left-looking' MGS implementation, we utilized \texttt{cuBLAS} routines \texttt{cublasDotEx}, \texttt{cublasAxpyEx}, and \texttt{cublasScaleEx}.
For `right-looking' MGS and CGS, we employed \texttt{cublasGemmEx} for efficient column updates. Note that we use \texttt{BLAS} level-3 routine rather than a Level-2 \texttt{GEMV}-based approach, primarily because \texttt{cuBLAS} does not provide a \texttt{GEMV} routine with the same flexibility in mixed-precision configurations as \texttt{cublasGemmEx}.
{We implemented custom \texttt{CUDA} kernels for most operations in the `right-looking' version of Hessenberg process.}
The first kernel is used to identify the index $\mathbf{\pi}_j$ of the element possessing the largest magnitude within the relevant sub-vector of a given column $j$ since the standard \texttt{cuBLAS} \texttt{cublasI<t>amax} routines lack \texttt{FP16} support. 
The resulting index $\mathbf{\pi}_j$ is stored directly in device memory.
Following the identification of $\mathbf{\pi}_j$, the $j$-th column is scaled based on the value of the element at this index. Subsequently in the `right-looking' version, using this scaled vector, we compute the necessary scaling weights required for updating subsequent columns. 
Finally, these computed weights are used to apply the transformation to all subsequent columns ($j+1$ to $n$) through a linear combination, which is executed by another custom \texttt{CUDA} update kernel.

It is important to note that while our custom kernels correctly implement the required functionality and demonstrate effective performance in our experiments, they were developed as prototypes. 
Unlike the highly tuned routines available in libraries such as \texttt{cuBLAS}, our kernels have not undergone extensive performance optimization. 
For instance, we employed fixed kernel block sizes across all tests and did not undertake architecture-specific tuning to identify the most efficient configurations for different GPUs. 
Also, we did not use block update as in the LU factorization routines discussed earlier based on \texttt{MAGMA} \cite{abdelfattah2024magma}.
Consequently, although the current implementation already achieves a notable level of performance, we anticipate that substantial further speedups could be realized through dedicated optimization efforts targeting these custom kernels.


\subsection{Approximation Accuracy}
{We first evaluated the accuracy of OFRR across three problem classes using \texttt{MATLAB}. For all experiments, we followed a consistent methodology. Each experiment was repeated with  three precisions: \texttt{FP64}, \texttt{FP32} and \texttt{FP16} for MatVecs and building bases.  Note that we always used double precision \texttt{FP64} when solving the final small (generalized) eigenvalue problems in the projected space. We used single precision to form this problem for the half-precision cases, as described in Figure~\ref{fig:diagram}. The tolerances to exclude columns were set dynamically relative to the machine epsilon of the respective precision, defined as $\varepsilon_{\texttt{FP64}}\approx2.22\times10^{-16}$ for \texttt{FP64}, $\varepsilon_{\texttt{FP32}}\approx1.19\times10^{-7}$ for \texttt{FP32}, and $\varepsilon_{\texttt{FP16}}\approx9.77\times10^{-4}$ for \texttt{FP16}. Random initial matrices/vectors with entries drawn from the uniform distribution $\mathcal{U}(0,1)$ were generated in FP64 and cast to the target precision; same sample was reused within each group of tests.
}

\subsubsection{Eigenvalue problems with kernel matrices}\label{sec:eig_kernel}

In our first set of experiments, we tested the performance of OFRR on eigenvalue problems with the Gaussian kernel matrices defined in Section~\ref{sec2} using a large length scale parameter $l$ and a small variance parameter $s$. 
This setup ensures that all the problems we test have only a few eigenvalues with large magnitudes.
Recall that for a dataset $\mathbf{D} \in \mathbb{R}^{n \times d}$, if we denote by $\mathbf{x}_i$ the $i$-th data point, the Gaussian kernel matrix is defined as $\mathbf{A}_{ij}= f(\exp(-\Vert\mathbf{x}_i-\mathbf{x}_j\Vert_2^2 / (2l^2))+s\delta_{ij})$, where $\delta_{ij}$ is a Kronecker delta function. 

In the first test, we uniformly sampled $1000$ data points from a square area with side length $\sqrt{1000}$, setting $f=0.2$, $l=10$, $s=0.01$, and $f=0.2$, $l=100$, $s=0$ to generate two test matrices.
The first test problem is not strictly numerically low-rank, as the eigenvalues only decay to $0.01$, as illustrated in {subplot (2,2)} of Figure~\ref{fig:test_20}. The second problem has a faster decay to $0.0$, as illustrated in {subplot (4,2)} of Figure~\ref{fig:test_20}.
We compared three different combinations of algorithms: the classical Rayleigh-Ritz projection with QR, OFRR with QR, and OFRR with the Hessenberg process.
We used MGS with re-orthogonalization to perform the QR factorization since it is the most accurate option.
For all the tests on the first matrix, we set the subspace dimension $k=50$, ran $m=10$ iterations with a step size $iter=3$, and reported the accuracy of the $20$ largest eigenvalues.
For the tests on the second matrix, since the eigenvalue decays faster, we set the subspace dimension $k=20$, ran $m=5$ iterations with a step size $iter=2$, and reported the accuracy of the $6$ largest eigenvalues. 

\begin{figure}[htbp]
    \centering
    \includegraphics[width=0.95\linewidth]{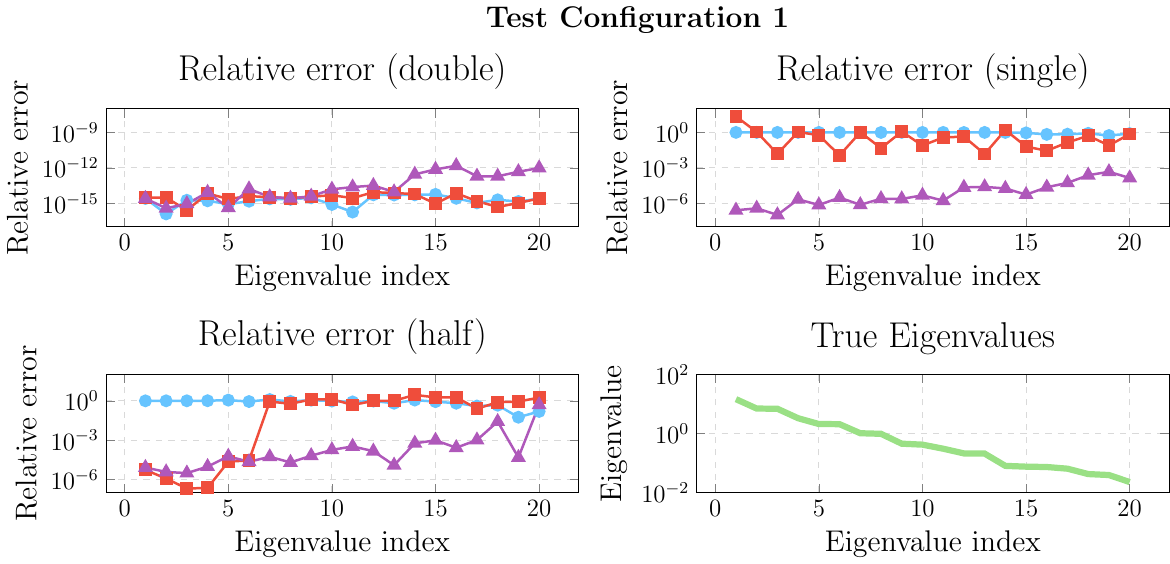}
    \includegraphics[width=0.95\linewidth]{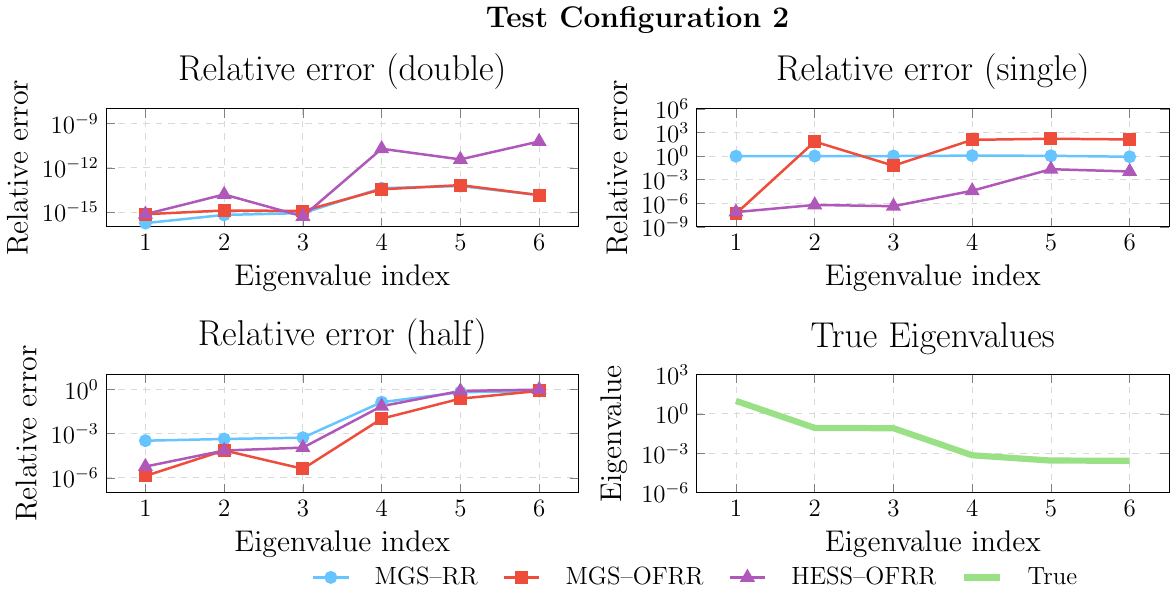}
    \caption{Relative approximation accuracy using different algorithms with different precisions and true leading eigenvalues. The test matrices are Gaussian kernel matrices of size $1000\times1000$ with $f=0.2$, $l=10$, $s=0.01$ ({test configuration 1}) and $f=0.2$, $l = 100$, $s = 0$ ({test configuration 2}). 
    }
    \label{fig:test_20}
\end{figure}

As we can see from Figure~\ref{fig:test_20}, under double precision, all three methods achieve high accuracy; the two QR‑based schemes are marginally superior because they employ orthonormal bases.
{Under single precision, both QR-based options exhibited relative errors larger than one. 
This deterioration arises from the loss of orthogonality in the single‑precision QR basis, which introduces several spurious large Ritz values and shifts the remaining eigenvalue approximations by at least one index.
{Under half precision, the use of a large tolerance $\varepsilon_{\texttt{FP16}}\approx 9.77\times10^{-4}$ eliminates more columns in the basis construction process, so the accuracy of the two QR-based algorithms both improved in the second test configuration.}
Even so, the orthogonal Rayleigh-Ritz-based method is still less accurate than OFRR with the Hessenberg process.}
On the other hand, even OFRR with QR produces results comparable to those of OFRR with the Hessenberg process in the second test configuration, the Hessenberg variant is significantly more efficient.

\subsubsection{Eigenvalue problems with sparse matrices}\label{sec:eig_sparse}

{In the next set of experiments, we evaluated the performance of several Krylov subspace methods within the OFRR framework, using sparse matrices from the Suite-Sparse matrix collection~\cite{davis2011university}. 
Specifically, we compared three algorithmic combinations: the classical Rayleigh-Ritz projection with the Lanczos method, the OFRR with the Lanczos method, and the OFRR with the Krylov-Hessenberg process.
Here, we use Lanczos with full orthogonalization, which is equivalent to the Arnoldi method for symmetric matrices.
For tests with restart turned on, we always restarted with the single Ritz vector corresponding to the largest Ritz value.
For these tests, we selected three matrices: \texttt{BCSSTK01}, \texttt{BCSSTK03}, and \texttt{1138\_BUS}.
Each matrix was scaled so that the largest eigenvalue is below $100$ to avoid overflow in half precision.
Key properties of these matrices are summarized in Table~\ref{tab:test_40_data}.
Because the Krylov subspace might not contain all leading eigenvectors, direct comparison of the computed Ritz values against the exact leading eigenvalues is not meaningful.
Instead, we assess the accuracy of computed approximate eigenpairs $(\lambda,\mathbf{v})$ by reporting the relative residual norm $\|\mathbf{A}\mathbf{v}-\lambda\mathbf{v}\|_2/|\lambda|$.
For \texttt{BCSSTK01} we set the Krylov subspace dimension to $20$ and report the $5$ largest Ritz pairs.  
For \texttt{BCSSTK03} the dimension is $50$ with $4$ restarts, and we report the $10$ largest Ritz pairs.  
For \texttt{1138\_BUS} the dimension is $100$ with $4$ restarts, and we report the $20$ largest Ritz pairs.}

\begin{table}[htbp]
\centering
\caption{Matrices from the SuiteSparse matrix collection. Here, $n$ is the matrix size, and $\mathrm{nnz}$ denotes the number of nonzeros. \label{tab:test_40_data}}
\begin{tabular}{c|ccc}
\toprule
 & \texttt{BCSSTK01} & \texttt{BCSSTK03} & \texttt{1138\_BUS} \\
\midrule
$n$    & $48$ & $112$ & $1128$ \\
$\mathrm{nnz}$  & $400$ & $640$ & $4054$ \\
Application  & Structural & Structural & Power Network \\
\bottomrule
\end{tabular}
\end{table}

\begin{figure}[htbp]
    \centering
    \includegraphics[width=0.95\linewidth]{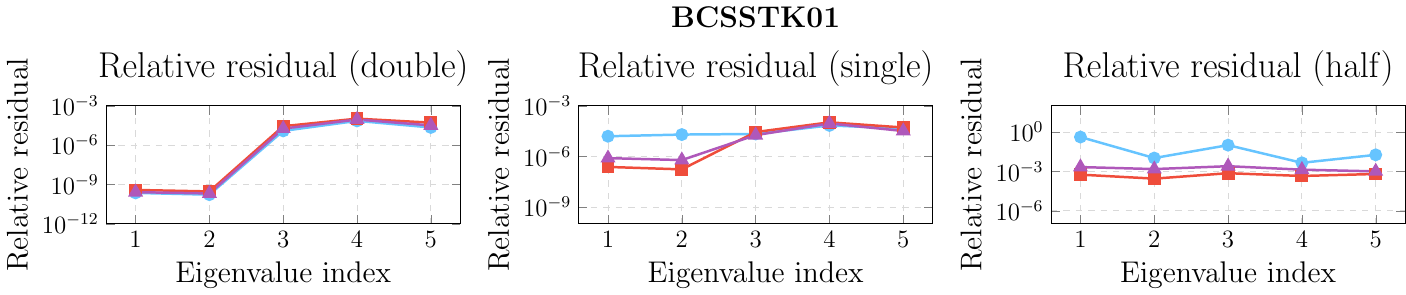}
    \includegraphics[width=0.95\linewidth]{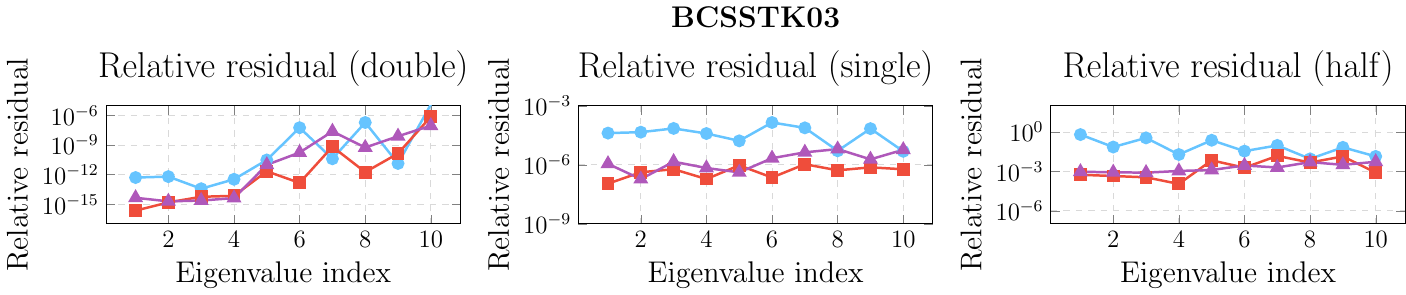}
    \includegraphics[width=0.95\linewidth]{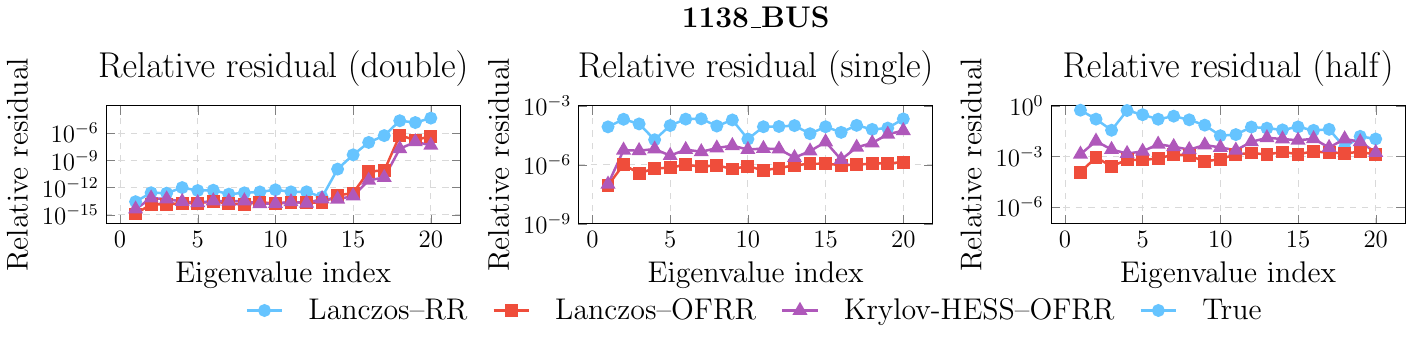}
    \caption{Relative residual norm using different algorithms with different precisions. The test matrices are from SuiteSparse matrix collection.
    }
    \label{fig:test_40}
\end{figure}

Figure~\ref{fig:test_40} confirms the trend established in Section~\ref{sec:eig_kernel}. All solvers achieve high accuracy in double precision. Consistent with results from previous half-precision tests, under single and half precision, the standard Rayleigh-Ritz method combined with Lanczos proved less accurate than the two methods utilizing the OFRR framework across all three test problems. Furthermore, comparing the two OFRR variants, the Hessenberg-based approach yielded an accuracy comparable to that of the Lanczos-based approach, reaffirming the benefits of  using the Hessenberg process with OFRR in reduced precision.

\subsubsection{Singular value decomposition with kernel matrices}\label{sec:svd_kernel}

{
In our next set of experiments, we tested the performance of OFRR on SVD with Gaussian kernel matrices.
{For the SVD experiments, we utilized the dataset comprising $n=1000$ data points previously generated for the eigenvalue tests presented in Section~\ref{sec:eig_kernel}. 
We denote this dataset as $\mathbf{D}_x \in \mathbb{R}^{n \times d}$. 
Subsequently, a second dataset, $\mathbf{D}_y \in \mathbb{R}^{m \times d}$ where $m=200$, was created by randomly selecting $m$ points from $\mathbf{D}_x$ without replacement.
We then constructed two $1000\times200$ kernel matrices $\mathbf{A}$ defined by $\mathbf{A}_{ij}= f(\exp(-\Vert\mathbf{x}_i-\mathbf{y}_j\Vert_2^2 / (2l^2)))$ for our test with $f=0.2$, $l=10$, and $f=0.2$, $l=100$.}
For both SVD test matrices, we performed $m=10$ iterations with a step size $iter=1$. For the first matrix, we used a subspace dimension of $k=20$ and reported the accuracy of the $10$ largest approximate singular values. For the second matrix, the subspace dimension was set to $k=10$, and we reported the accuracy of the $5$ largest approximate singular values.

\begin{figure}[htbp]
    \centering
    \includegraphics[width=0.95\linewidth]{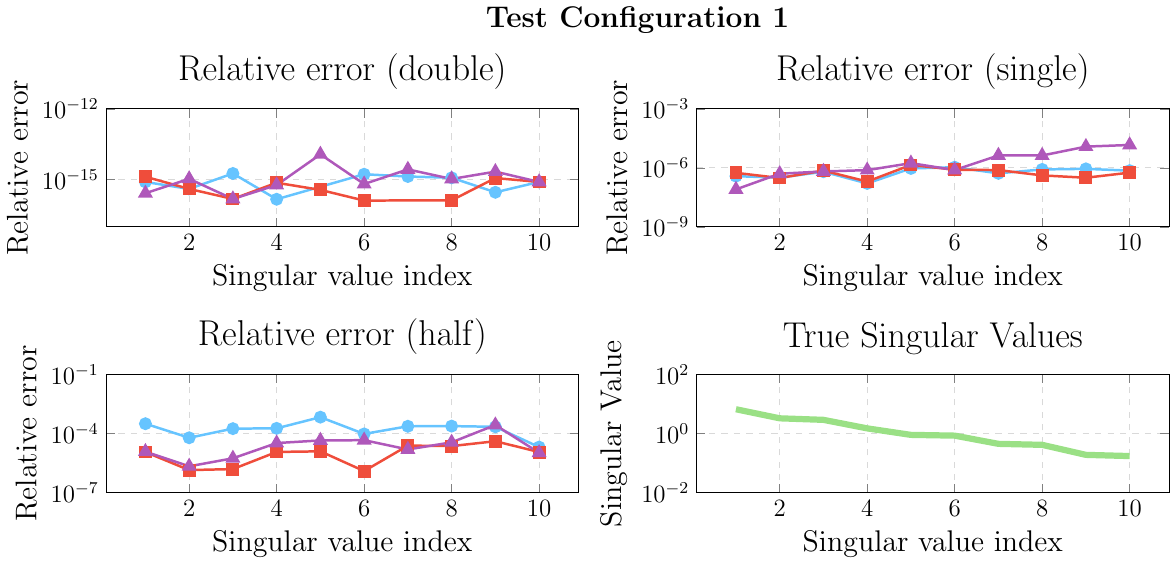}
    \includegraphics[width=0.95\linewidth]{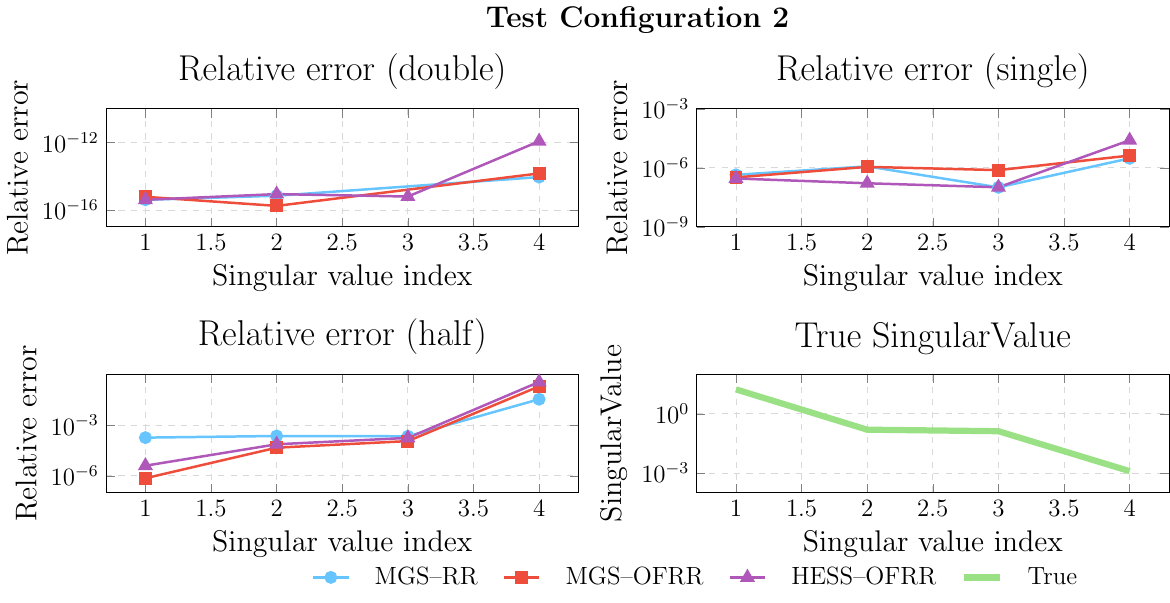}
    \caption{Relative approximation accuracy using different algorithms with different precisions (left) and true leading singular values (right). The test matrices are Gaussian kernel matrices of size $1000\times200$ with $f=0.2$, $l=10$ ({test configuration 1}) and $f=0.2$, $l = 100$ ({test configuration 2}). 
    }
    \label{fig:test_30}
\end{figure}

The results for the SVD approximation tests are presented in Figure~\ref{fig:test_30}. 
In both double precision and single precision, all algorithmic approaches provided accurate approximations of the singular values, with the two QR-based methods exhibiting slightly higher accuracy due to the use of orthogonal bases.
In half precision, the results continued to align with the primary trend observed in earlier experiments. The two methods based on the OFRR framework achieved better accuracy than the Rayleigh-Ritz-based method similar to the eigenvalue tests. Furthermore, within the OFRR framework, the Hessenberg variant yielded accuracy comparable to the QR variant in half precision.
This similarity in  the achieved accuracies, combined with the known computational advantages of the Hessenberg process, 
bolsters  its appeal as a tool for OFRR in low-precision computations.
}

\subsection{Speedup}

Having validated the numerical accuracy using the \texttt{MATLAB} implementation, we now shift our focus to quantifying the computational performance of key algorithmic components. 
To this end, we leverage the C++/\texttt{CUDA} implementations developed specifically for execution on GPU architectures.

This section focuses on analyzing the performance of routines responsible for generating linearly independent bases, as this step is a core computational kernel in the OFRR framework. Whether OFRR is integrated with Arnoldi-type iterations or embedded in subspace iteration schemes for solving eigenvalue or SVD problems, the basis generation step remains the dominant performance-critical component.
Therefore, we specifically compare the runtime performance of the following algorithms designed for this task: (i) Left-looking Modified Gram-Schmidt without re-orthogonalization (\texttt{MGS-L}); (ii) Right-looking Modified Gram-Schmidt (\texttt{MGS-R}); (iii) Left-looking Classical Gram-Schmidt without re-orthogonalization (\texttt{CGS}); (iv) Left-looking Classical Gram-Schmidt with re-orthogonalization (\texttt{CGS-2}); (v) Left-looking Hessenberg Process (\texttt{Hess-L}); (vi) Right-looking Hessenberg Process (\texttt{Hess-R}).
Benchmarking the efficiency of these routines provides a direct and representative assessment of the overall performance of the OFRR framework, without the confounding effects of outer iteration strategies or back-end solvers, which typically rely on standard, highly optimized \texttt{BLAS} libraries.

The tests utilized three distinct fixed-size input matrices with entries drawn from the uniform distribution $\mathcal{U}(0,1)$ with dimensions $25000\times200$, $50000\times200$, and $50000\times400$, and were conducted under \texttt{FP64}, \texttt{FP32}, and \texttt{FP16} precision. For \texttt{FP16}, all computations were internally carried out in \texttt{FP32}. Each test was repeated five times, and the average runtime is reported.

\begin{table}[htbp]
    \centering
    \caption{Runtime (in milliseconds) of MGS (without re-orthogonalization), CGS, and Hessenberg process on GPU across multiple precisions and matrix sizes. Here, ``-L'' and ``-R'' denote left- and right-looking variants, respectively, and ``\texttt{CGS-2}'' indicates CGS with re-orthogonalization. For \texttt{FP16}, all computations were internally carried out in \texttt{FP32}.}
    \label{tab:basis_all_precisions}
    \begin{tabular}{ll|ccc}
        \toprule
        \multirow{2}{*}{Precision} & \multirow{2}{*}{Method} & \multicolumn{3}{c}{Matrix Dimensions ($m \times n$)} \\
        \cmidrule(lr){3-5}
         & & $25000\times200$ & $50000\times200$ & $50000\times400$ \\
        \midrule
        \multirow{6}{*}{\texttt{FP64}}
            & \texttt{MGS-L}   & 359.805  & 369.862  & 1402.376 \\
            & \texttt{MGS-R}   & 42.936   & 67.040   & 219.855   \\
            & \texttt{CGS}      & 46.322   & 65.882   & 182.747  \\
            & \texttt{CGS-2}     & 55.346   & 82.591   & 243.225  \\
            & \texttt{Hess-L}  & 119.293  & 132.390 & 480.391  \\
            & \texttt{Hess-R}  & 24.677   & 39.528  & 138.033   \\
        \midrule
        \multirow{6}{*}{\texttt{FP32}}
            & \texttt{MGS-L}   & 334.379  & 335.997  & 1331.892 \\
            & \texttt{MGS-R}   & 23.062   & 27.906   & 99.559   \\
            & \texttt{CGS}      & 14.824   & 21.894  & 72.658  \\
            & \texttt{CGS-2}      & 18.987   & 30.062   & 102.819 \\
            & \texttt{Hess-L}  & 111.796  & 121.182 & 446.195  \\
            & \texttt{Hess-R}  & 16.571   & 24.286  & 77.928   \\
        \midrule
        \multirow{6}{*}{\texttt{FP16}}
            & \texttt{MGS-L}   & 334.164  & 335.369  & 1328.612 \\
            & \texttt{MGS-R}   & 29.499   & 17.821   & 57.887   \\
            & \texttt{CGS}      & 12.971   & 15.894   & 47.782  \\
            & \texttt{CGS-2}      & 15.025   & 20.742   & 64.865  \\
            & \texttt{Hess-L}  & 111.314  & 120.955 & 441.900  \\
            & \texttt{Hess-R}  & 12.754   & 17.273  & 49.223   \\
        \bottomrule
    \end{tabular}
\end{table}

Table~\ref{tab:basis_all_precisions} presents the runtimes (in milliseconds) of various basis generation methods across three matrix sizes and three floating-point precisions on GPU. We first observe a consistent performance advantage for right-looking algorithms (``-R'') over their left-looking counterparts (``-L''). This is especially pronounced for MGS: \texttt{MGS-R} achieves more than $10\times$ speedup over \texttt{MGS-L} across all tested sizes, confirming that right-looking structures are significantly more GPU-friendly due to better memory access and data locality.

Comparing the Hessenberg process to MGS, we find that right-looking Hessenberg (\texttt{Hess-R}) achieves comparable or better performance than \texttt{MGS-R} in most configurations. For instance, at size $50000\times400$ under \texttt{FP32}, \texttt{Hess-R} completes in 77.9ms versus 99.6ms for \texttt{MGS-R}. This is particularly encouraging given that the Hessenberg implementation relies on a custom GPU kernel, which has not yet been fully optimized. Further performance gains are expected with improved kernel-level optimizations, including memory fusion, architecture-aware block sizing, and better occupancy tuning.

The performance advantage of the Hessenberg process over MGS also holds for their left-looking variants. Across all tested sizes and precisions, \texttt{Hess-L} consistently outperforms \texttt{MGS-L}, with speedups ranging from 2--3$\times$ at larger problem scales. This improvement stems largely from the inner-product-free nature of the Hessenberg process. Moreover, while CGS is often considered a more efficient alternative to MGS in left-looking settings, its numerical instability under finite precision can be problematic. For instance, Figure \ref{fig:test_00} demonstrates a case where the instability of CGS impacts the condition number of the basis, highlighting the advantage of \texttt{Hess-L}. In such cases, \texttt{Hess-L} offers both better runtime and improved robustness.

To analyze scaling with respect to matrix dimensions, we compare cases with increasing $m$ and $n$. Doubling $m$ (e.g., $25000\times200$ to $50000\times200$) leads to negligible runtime growth for left-looking methods like \texttt{MGS-L}, reflecting the limited parallelism of reduction-based operations such as inner products. In contrast, doubling $n$ (e.g., $50000\times200$ to $50000\times400$) yields near $4\times$ runtime increase, consistent with the $O(mn^2)$ cost.

Precision-wise, we observe meaningful runtime reductions from \texttt{FP64} to \texttt{FP32} and \texttt{FP16}, but the improvement is highly method-dependent. Right-looking methods benefit most from reduced precision, with \texttt{Hess-R} and \texttt{MGS-R} showing clear speedups. In contrast, left-looking methods, especially \texttt{MGS-L}, show little to no performance gain. For instance, \texttt{MGS-L} takes 334ms under both \texttt{FP32} and \texttt{FP16} at $25000\times200$, essentially unchanged from its \texttt{FP64} time. This can be attributed to the reliance on \texttt{BLAS} level-1 operations, which are limited by memory bandwidth and cannot exploit the arithmetic acceleration from lower-precision units. 

\section{Conclusion}\label{sec6}
In this paper, we investigated the use of the non-orthogonal Rayleigh–Ritz projection method for computing selected eigenvalues and singular values under varying levels of numerical precision. Our study highlights the advantages of the Hessenberg process in low-precision settings, as an alternative to the Modified Gram–Schmidt (MGS) procedure. Specifically, the Hessenberg process demonstrates not only competitive accuracy but also improved efficiency in GPU implementation. While our current implementation of the Hessenberg process already delivers competitive performance, it also reveals untapped potential for further optimization—particularly through the development of custom GPU kernels. These findings position the Hessenberg-based OFRR framework as a promising direction for developing efficient and scalable eigensolvers on modern hardware.



\bibliographystyle{siamplain}
\bibliography{papers}

\end{document}